\documentclass[11pt]{article}
 \usepackage{geometry}               
 \usepackage{amssymb}
 \usepackage{epstopdf}
 \usepackage{amsmath} 
 \usepackage{amsfonts}   
\usepackage{graphicx}   
\usepackage{verbatim}   
\usepackage{color, subfigure}  
\usepackage{hyperref}   
\usepackage{amscd}
\usepackage{amsthm}
\usepackage{makeidx}
\usepackage{epsfig}

\newtheorem{theorem}{Theorem}[section]

\newtheorem{theorem-definition}[theorem]{Theorem-Definition}
\newtheorem{theorem-construction}[theorem]{Theorem-Construction}
\newtheorem{lemma-definition}[theorem]{Lemma--Definition}
\newtheorem{lemma-construction}[theorem]{Lemma--Construction}
\newtheorem{lemma}[theorem]{Lemma}
\newtheorem{proposition}[theorem]{Proposition}
\newtheorem{corollary}[theorem]{Corollary}
\newtheorem{conjecture}[theorem]{Conjecture}
\newtheorem{name}[theorem]{Definition}

\newcommand{\Z}{{\mathbb Z}}
\newcommand{\R}{{\mathbb R}}
\newcommand{\Q}{{\mathbb Q}}
\newcommand{\C}{{\mathbb C}}

\newcommand{\lms}{\longmapsto}
\newcommand{\lra}{\longrightarrow}

\newcommand{\goto}{\rightarrow}
\newcommand{\be}{\begin{equation}}
\newcommand{\ee}{\end{equation}}
\newcommand{\bt}{\begin{theorem}}

\newcommand{\et}{\end{theorem}}
\newcommand{\bd}{\begin{name}}
\newcommand{\ed}{\end{name}}
\newcommand{\bp}{\begin{proposition}}
\newcommand{\ep}{\end{proposition}}

\newcommand{\bl}{\begin{lemma}}
\newcommand{\el}{\end{lemma}}
\newcommand{\bc}{\begin{corollary}}
\newcommand{\ec}{\end{corollary}}
\newcommand{\bcon}{\begin{conjecture}}
\newcommand{\econ}{\end{conjecture}}

\newcommand{\la}{\label}
\newcommand{\mc}{\mathcal}
 
\newcommand{\eqarray}[3]{\begin{equation} #1=\left\{ \begin{array}{#2} #3 \end{array} \right. \end{equation}}\newcommand{\earray}[2]{\begin{equation} \quad \left\{ \begin{array}{#1} #2 \end{array} \right. \end{equation}}
\newcommand{\Earray}[2]{\begin{equation} \begin{array}{#1} #2 \end{array} \end{equation}}

\newcommand{\disp}{\displaystyle}
\newcommand{\imap}[1]{\mathbb{I}_{\mc{#1}}}
\newcommand{\pr}[1]{\text{$\text{$\bf P$}$} \rm{(\mathcal{#1})}}
\newcommand{\tropspa}[1]{\mathcal{#1}(\mathbb{A}^t)}
 
\newcommand{\tropsp}[1]{\mathcal{#1}(\mathbb{Z}^t)}
\newcommand{\tropspr}[1]{\mathcal{#1}(\mathbb{R}^t)}
\newcommand{\exset}[1]{\text{$\text{$\bf E$}$} \rm{(\mathcal{#1})}}

\begin{document}

\title{Stasheff polytopes and the coordinate ring of the cluster $\mathcal{X}$-variety of type $A_n$}
\author{Linhui Shen}
\date{}
\maketitle
\begin{abstract}{We define Stasheff polytopes in the spaces of tropical points of cluster $\mathcal{A}$-varieties. We study the supports of products of elements of canonical bases for cluster $\mathcal{X}$-varieties. We prove that, for the cluster $\mathcal{X}$-variety of type $A_n$, such supports are Stasheff polytopes.}
\end{abstract}
 
\numberwithin{equation}{section}

 \tableofcontents

\section{Introduction}
A cluster ensemble is a pair $({\cal X}, {\cal A})$ of positive spaces defined in~\cite{FG}. Cluster $\mc{A}$-varieties are closely related to cluster algebras introduced in \cite{FZ} -- the ring of regular functions on the space $\mc{A}$ coincides with the upper cluster algebra of \cite{BFZ}. In this paper, we focus on the cluster $\mc{X}$-variety $\mc{X}_{A_n}$ of type $A_n$, which is closely related to the moduli space $\mc{M}_{0, n+3}$.  

In detail, let $(\mc{A}_\Phi ,\mc{X}_\Phi)$ be the cluster ensemble assigned to a root system $\Phi$ of finite type. Denote by $\mc{A}_{{\Phi}^{\vee}}(\mathbb{Z}^t)$ the space of $\mathbb{Z}$-\emph{tropical~points} of the Langlands dual cluster $\mc{A}$-variety $\mc{A}_{\Phi^{\vee}}$. According to the Duality Conjectures from \cite{FG}, one should have an isomorphism 
\begin{equation}
\imap{A}:  \quad {\mc{A}}_{\Phi^{\vee}} (\mathbb{Z}^t) \longrightarrow {\bf E}({\mc{X}_{\Phi}})
\end{equation}
where ${\bf E}({\mc X}_{\Phi})$ is the set of 
\emph{indecomposable~functions} on ${\mc X}_{\Phi}$, providing a canonical basis of
the ring of regular functions on ${\mc X}_{\Phi}$. In the case of type $A_n$, such a map $\imap{A}$ has been constructed in \cite{FG3}.\footnote{In the case of type $A_n$, the Langlands dual $\mc{A}_{{A_n}^\vee}$ coincides with $\mc{A}_{A_n}.$} Denote by ${\bf P}(\mc{X}_{A_n})$ the set of finite products of elements of ${\bf E}(\mc{X}_{A_n})$. The main goal of this paper is to study ${\bf P}(\mc{X}_{A_n})$. 

The basis ${\bf E}(\mc{X}_{A_n})$ has remarkable combinatorial properties. One of them is that, given tropical points $l_1, ..., l_m \in {\mc A}_{A_n}(\mathbb{Z}^t)$, their corresponding product in $ {\bf P}(\mc{X}_{A_n}) $ can be uniquely decomposed into a finite sum of elements of ${\bf E}(\mc{X}_{A_n})$ with {\it non-negative} coefficients:
\begin{equation}
\disp \prod_{i=1}^m \imap{A}(l_i) = 
\sum_{l\in {\mc{A}}_{A_n}(\mathbb{Z}^t)}c(l_1, ..., l_m; l)\imap{A}(l).
\label{imap}
\end{equation}
The \emph{support~of~the~product} is the set of $l \in {\mc{A}}_{A_n}(\mathbb{Z}^t)$ 
such that  $c(l_1, ..., l_m; l) \not = 0$. The Convexity Conjecture from \cite{FG} suggests that the support is \emph{universally~convex} in $\mc{A}_{A_n}(\mathbb{Z}^t)$. We prove that it is not only convex but also a Stasheff polytope.

\vskip 2mm

The Stasheff polytope is a remarkable convex polytope first described combinatorially by J. Stasheff in 1963.
  In Section~\ref{sec: 2}, we define Stasheff polytopes in \emph{tropical~positive~spaces}. The main idea of the definition is given below.
 
Denote by $\mc{A}(\mathbb{R}^t)$ the space of $\mathbb{R}$-tropical points of a positive space $\mc{A}$. It has a piecewise-linear structure, isomorphic to $\mathbb{R}^{\dim\mc{A}}$ in many different ways.  We define a \emph{convex~polytope} in $\tropspr{A}$ as the intersection of ``half-spaces" given by inequalities of \emph{tropical~indecomposable~functions} (see Section~\ref{sec: 1.1}). The tropical indecomposable functions are usually not linear in all coordinate systems, but are convex and piecewise linear. 
 
Let $\Phi_{\geq -1}$ be the set of positive and simple negative roots  of a root system $\Phi$ of finite type. In \cite{CFZ}, polytopal realizations of Stasheff polytopes (or generalized associahedra) in $\mathbb{R}^n$ were constructed by a set of linear inequalities indexed by $\Phi_{\geq-1}$.  For the cluster $\mc{A}$-variety $\mc{A}_{\Phi}$, there are  cluster variables $A_\alpha$ indexed by $\Phi_{\geq -1}$ ([FZ1, Theorem 5.7]). When $\Phi$ is of classical Cartan-Killing type, the cluster variables are indecomposable functions ([\emph{loc.cit.}, Theorem 4.27]). Their \emph{tropicalizations} $A^t_\alpha$ are functions on the space $\mc{A}_\Phi(\mathbb{Z}^t)$. Our Stasheff polytopes in $\mc{A}_{\Phi}(\mathbb{Z}^t)$ are defined via tropical cluster variables (Definition \ref{eq32}). One of our main results is as follows.

\bt \la{mthm}
The support of the product \eqref{imap} is the Minkowski sum of 
the points $l_1, ..., l_m$: 
\begin{equation}
S(l_1, ..., l_m) = 
\{x \in {\mc{A}}_{A_n}(\mathbb{Z}^t)~|~A^t_\alpha(x) \leq \sum_{i=1}^m A^t_\alpha(l_i) ~~\text{for any $\alpha \in \Phi_{\geq -1}$}\}.
\end{equation}
It is a Stasheff polytope in the space 
${\mc A}_{A_n}(\mathbb{Z}^t)$. 
\et

We conjecture the same result for any root system $\Phi$ of finite type. 

In $\mathbb{R}^n$ the Minkowski sum of finitely many points is the sum of these points as vectors. It is still a point.

In our case, ${\cal A}_{A_n}$ has many coordinate systems which are called \emph{clusters}. Each cluster provides a way of taking the sum of the tropical points as vectors. We thus get a set of vertices parametrized by the set of clusters. The Minkowski sum turns out to be the convex hull of these vertices. It is a Stasheff polytope. 

An example of Stasheff polytopes in $\mc{A}_{A_2}$ is shown on Fig.~\ref{fig1}-\ref{fig5}. These figures show the same polytope in 5 different coordinate systems (clusters). Notice that it is a ``tropical" pentagon, but looks like a heptagon from the usual point of view: the reason is that some of its sides are given by piecewise linear functions, and thus they are not line segments. But they become line segments in other coordinate systems. In general, the Minkowski sum gives rise to a new family of convex polytopes in tropical positive spaces (Definition \ref{eq88}).

\vskip 2mm

Section \ref{sec: 3} focuses on the cluster $\mc{A}$-variety $ \widetilde{\cal A}_{A_n}$ . It is a cluster ${\cal A}$-variety with coefficients (see Definition \ref{eq35}). Theorems \ref{eq42}, \ref{eq55} provide two criteria for recognizing Stasheff polytopes. Theorem \ref{6.3.12.44} is a corollary of Theorem \ref{eq42}. The variety $\widetilde{\cal A}_{A_n}$ also has a canonical basis ${\bf E}(\widetilde{\cal A}_{A_n})$~(\cite{FZ1}, \cite{SZ},\cite{C}). Similarly, we define the set ${\bf P}(\widetilde{\mc{A}}_{A_n})$ of finite products of elements of  ${\bf E}(\widetilde{\mc{A}}_{A_n})$. It has a natural partial order structure. Theorem \ref{thm2.4} determines the partial order on  ${\bf P}(\widetilde{\mc{A}}_{A_n})$. Theorem~\ref{eq61} is the main technical tool for proving Theorem \ref{mthm}.

\vskip 2mm

Section \ref{sec: 4} focuses on the cluster $\mc{X}$-variety $\mc{X}_{A_n}$. There is a surjective map $k:\widetilde{\cal A}_{A_n}\goto {\cal X}_{A_n}$. The induced map $k^*$ on their coordinate rings takes ${\bf E}(\mc{X}_{A_n})$ (respectively ${\bf P}(\mc{X}_{A_n})$) into ${\bf E}(\widetilde{\mc{A}}_{A_n})$ (respectively ${\bf P}(\widetilde{\mc{A}}_{A_n})$). By using Theorem \ref{eq61} and the map $k$, we prove the first part of Theorem \ref{mthm}. The second part is a direct consequence of Theorem \ref{6.3.12.44}. Theorem \ref{thm: 4.2} provides a bijection between ${\bf P}_{*}(\mc{X}_{A_n})$ and the set of Stasheff polytopes in $\mc{A}_{A_n}(\mathbb{Z}^t)$. Conjecture \ref{eq91} is a generalization of Theorem \ref{6.3.12.44} to all cluster $\mc{X}$-varieties.
  
 \vskip 3mm
  
 {\bf Acknowledgements}. I wish to thank K. Peng, R. Raj for helpful conversations. I am especially grateful to my advisor A. Goncharov for suggesting the problem and for his enlightening suggestions and encouragement. In particular, the main definitions of this paper concerning convex polytopes in tropical positive spaces follow the idea in [FG3, Section 2.4]. Finally, I thank the referee for very careful reading of this paper and for many useful suggestions.

\subsection{Tropical positive spaces and convex subsets}
\label{sec: 1.1}
For the convenience of the reader, let us briefly recall some basic definitions from \cite{FG}.

A \emph{positive~space} is a variety $\mc{A}$ equipped with a positive atlas $C_{\mc{A}}$. Namely, the transition maps between coordinate systems in $C_{\mc{A}}$ are given by rational functions presented as a ratio of two polynomials with positive integral coefficients. We denote such a space by ($\mc{A}$, $C_{\mc{A}}$). 

A \emph{universally positive Laurent polynomial} on $\mc{A}$ is a regular function on $\mc{A}$ which is a Laurent polynomial with non-negative integral coefficients in every coordinate system in $C_{\mc{A}}$. Denote by $\mathbb{L}_{+}(\mc{A})$ the set of universally positive Laurent polynomials.  Given $F_1, F_2\in \mathbb{L}_{+}(\mc{A})$, if $F_1-F_2\in \mathbb{L}_{+}(\mc{A})$, then we say $F_1\geq F_2$. We say that $F\in \mathbb{L}_{+}(\mc{A})$ is \emph{indecomposable} if it cannot be decomposed into a sum of two nonzero universally positive Laurent polynomials. Denote by $\exset{A}$ the set of indecomposable functions. Let $\pr{A}$ be the set of finite products of indecomposable functions. Clearly, $\pr{A}$ is a subset of  $\mathbb{L}_{+}(\mc{A})$, and is a semigroup under multiplication.

\vskip 3mm

A \emph{semifield} is a set \emph{P} equipped with operations of addition and multiplication, so that addition is commutative and associative, multiplication makes \emph{P} into an abelian group, and they are compatible in a natural way: $(a+b)c=ac+bc$ for $a,b,c \in P$.  For any positive space $\mc{A}$, the transition maps are subtraction free. Thus one can consider the set ${\cal A}(P)$ of $P$-points of $\mc{A}$. The transition maps are bijective on ${\cal A}(P)$ because they are well defined on every \emph{P}-point. Therefore ${\cal A}(P)\simeq P^n$. For example, the set ${\R_{>0}}$ of positive real numbers with usual operations is a semifield. The set ${\cal A}(\R_{>0})\simeq \R_{>0}^n$ of positive points of ${\cal A}$ is well defined.

The tropical semifield $\mathbb{R}^t$ is a set of real numbers $\mathbb{R}$ but with the multiplication $\cdot_t$ and addition $+_t$ given by
  \be \la{eq2}
   a\cdot_t b:=a+b, \quad a+_t b := \max\{a,b\}.
  \ee
The semifields $\mathbb{Z}^t$, $\mathbb{Q}^t$ are defined in the same way. Let $\mc{A}(\mathbb{A}^{t})$ be the set of $\mathbb{A}^t$-points, here $\mathbb{A}$ can be $\mathbb{Z}$, $\mathbb{Q}$, $\mathbb{R}$. One can tropicalize $F\in\mathbb{L}_{+}(\mc{A})$ via evaluating it on $\mc{A}(\mathbb{A}^t)$. It is easy to see that the tropicalization $F^t$ is a convex piecewise linear function in each positive coordinate system. 

{\bf Example.} Let $F=2x_1^3x_2+x_1^{-1}+1=x_1^3x_2+x_1^3x_2+x_1^{-1}+1$. When taking the maximum, the coefficients of monomials in \emph{F} do not matter. One can drop them first. Hence $F^t=\max\{3x_1+x_2, -x_1,0\}.$

\vskip 3mm

Given a subset \emph{S} of $\mc{A}(\mathbb{A}^t)$, define
\be \la{eq4}
c_{F,S}=\sup_{x\in S}F^t(x), \quad c_{F,S} \in \mathbb{A}\bigcup\{+\infty\}.
\ee
Following \cite{FG3}, the convex hull of $S$ and the Minkowski sum of two convex subsets are as follows.
\bd \la{eq5}
The convex hull of a subset {S} of $\mc{A}(\mathbb{A}^t)$ is
\be \la{eq6}
C(S)=\{x\in\mc{A}(\mathbb{A}^t)~|~  F^t(x)\leq c_{F,S} \text{ for all } F\in \exset{A}\}.
\ee
Clearly $S \subseteq C(S)$. We say {S} is convex if $S=C(S)$.
\ed
\bd \la{eq7}
The Minkowski sum of two convex subsets $S_1$, $S_2$ of $\tropspa{A}$ is
\be \la{eq8}
S_1+S_2=\{x\in \tropspa{A}~|~ F^t(x)\leq c_{F,S_1}+c_{F,S_2} \text{ for all }F\in \exset{A}\}.
\ee
\ed
${\bf Remark.}$ Clearly $S_1+S_2$ is also convex, but the number
\be \la{eq9}
c_{F, S_1+S_2}=\sup_{x\in S_1+S_2}F^t(x)
\ee 
is not necessarily equal to $c_{F, S_1}+c_{F,S_2}$. Furthermore, the Minkowski sum may not be associative. This is because in our definition, the defining functions may not be linear. For example, let $L_1=x$, $L_2=y$, $L_3=\max\{x+y,x\}$ be functions on $\mathbb{R}^2$. Given three sets
$$S_1=\{L_1\leq 0, L_2 \leq 1, L_3\leq 1 \},\quad S_2=\{L_1\leq 1, L_2 \leq -1, L_3\leq 1\},$$$$S_3=\{L_1\leq 0, L_2\leq 1, L_3 \leq 0\}.$$
Consider $S_1+S_2=\{L_1\leq1, L_2\leq 0, L_3 \leq 2\},$ then $c_{L_3, S_1+S_2}=1<c_{L_3, S_1}+c_{L_3, S_2}$  and $(S_1+S_2)+S_3\neq S_1+(S_2+S_3).$

\vskip 3mm

In the cases of cluster $\mc{A}$-varieties of classical Cartan-Killing type (see \cite{FZ1}, Section 4), the Minkowski sum behaves well in the following sense.
\bp \la{eq10}
Given a cluster $\mc{A}$-variety $\mc{A}$ of classical type, let $S_1$, $S_2$, $S_3$ be convex subsets of $\mc{A}(\mathbb{A}^t)$. For any $F\in\exset{A}$, we have
\be \la{eq11}
c_{F,S_1+S_2}=c_{F,S_1}+c_{F,S_2}.
\ee 
The associativity holds:
\be \la{eq12}
(S_1+S_2)+S_3=S_1+(S_2+S_3).
\ee
\ep
\begin{proof}
For any $F\in{\bf E}({\cal A})$, it becomes a monomial in a certain coordinate system ([\emph{loc.cit.}, Theorem 4.27], [C, Theorem 1.1]). Its tropicalization $F^t$ is linear in this coordinate system. Let $x_1\in S_1, x_2\in S_2$ be such that $F^t(x_1)=c_{F,S_1}$, $F^t(x_2)=c_{F,S_2}$. Let $x_1+x_2$ be the usual sum of two vectors in the same coordinate system. Then $F^t(x_1+x_2)=c_{F,S_1}+c_{F,S_2}$. On the other hand, for any $G\in \exset{A}$, the convexity of $G^t$ implies that
$G^t(x_1+x_2)\leq G^t(x_1)+G^t(x_2)\leq c_{G,S_1}+c_{G,S_2}$.
Therefore, by definition, $x_1+x_2 \in S_1+S_2$. The number $c_{F,S_1+S_2}\geq F^t(x_1+x_2)=c_{F,S_1}+c_{F,S_2}$. The other direction that $c_{F, S_1+S_2}\leq c_{F,S_1}+c_{F,S_2}$ follows from (\ref{eq8}), (\ref{eq9}). The first part is proved. The associativity follows directly from the first part.
\end{proof}
\begin{conjecture} \la{eq14}
The formula \eqref{eq11} holds for all cluster $\mc{A}$-varieties.
\end{conjecture}
For cases when \eqref{eq11} holds, given finitely many subsets $S_1, ..., S_m$ of $\tropspa{A}$, their Minkowski sum is
\be \la{eq15}
\sum_{i=1}^mS_i=\{x\in \tropspa{A}~|~ F^t(x)\leq \sum_{i=1}^mc_{F,S_i} \text{ for all } F\in \exset{A} \}.
\ee

\subsection{Supports of products of elements of a canonical basis}
\la{sec: 1.2}
Let $(\mc{X}, \mc{A})$ be a cluster ensemble. We briefly recall its definition in Section \ref{sec: 4.1}. Let $({\cal X}^{\vee}, {\cal A}^{\vee})$ be the pair of their \emph{Langlands~dual~spaces} (\cite{FG}).
Fock and Goncharov's Duality Conjecture asserts that the set of $\mathbb{Z}^t$-points of $\mc{A}^{\vee}$-(or $\mc{X}^{\vee}$-) space parametrizes a canonical basis of the coordinate ring of  ${\cal X}$-(or ${\cal A}$-) space. For our purposes, we need only one direction of this conjecture.
\bcon [{[\emph{loc.cit}]}, Section 4] There is a canonical isomorphism
\be \la{eq18}
{\Bbb I}_{\cal A}: {\cal A}^{\vee}(\Z^t)\stackrel{\sim}{\lra} {\bf E}({\cal X}).
\ee
The set ${\bf E}({\cal X})$ provides a ${\Z}$-basis of the coordinate ring of ${\cal X}$. 
\econ

For  the cluster $\mc{X}$-variety $\mc{X}_{A_n}$, the Duality Conjecture is proved in \cite{FG3}. We sketch the proof in Section \ref{sec: 4}. In this case, ${\cal A}_{A_n}^{\vee}$ coincides with ${\cal A}_{A_n}$. Let \emph{f} be a regular function on ${\cal X}_{A_n}$. It can be uniquely decomposed as a finite sum
\be \la{eq19}
f=\sum_{l\in {\cal A}_{A_n}(\Z^t)}c(f;l)\imap{A}(l).
\ee
The numbers $c(f;*)$ are called the structure coefficients of \emph{f}. Thanks to the following Lemma, we have $c(f;*)\in \mathbb{Z}_{\geq 0}$ for each $f\in {\Bbb L}_{+}({\cal X}_{A_n})$.

\bl If the set ${\bf E}({\cal X})$ of indecomposable functions is a $\Z$-basis of the coordinate ring of ${\cal X}$, then it provides a $\Z_{\geq 0}$-basis of ${\Bbb L}_+({\cal X})$.
\el
\begin{proof}
Let $\alpha=\{X_i\}$ be a local coordinate system in $C_{\cal X}$. By definition, each $f\in {\Bbb L}_+({\cal X})$ can be uniquely expressed as
$$
f=\sum_{a=(a_1,\ldots, a_n)\in \Z^n} c_{a} X_1^{a_1}\ldots X_n^{a_n}, ~~~c_a\in \Z_{\geq 0}.
$$
Define $l_{\alpha}(f):=f(1,\ldots,1)=\sum_{a\in \Z^n} c_a$. Clearly here if $f\neq 0$, then $l_{\alpha}(f)\geq 1$.

Let $f \in {\Bbb L}_{+}({\cal X})-\{0\}$. Let $f=\sum_{k=1}^m g_k$ be a decomposition of $f$ such that $g_1,\ldots, g_m \in {\Bbb L}_+({\cal X})-\{0\}$. Then $l_{\alpha}(f)=\sum_{k=1}^m l_{\alpha}(g_k)\geq m$. Therefore $m$ is bounded. Let $N$ be the maximal number such that
\be \la{last}
f=g_1+\ldots +g_N, ~~~~ g_1,\ldots, g_N \in {\Bbb L}_+({\cal X})-\{0\}.
\ee
Here $g_1,\ldots, g_N \in {\bf E}({\cal X})$. Otherwise, we may assume that $g_N \notin {\bf E}({\cal X})$. By definition, we have $g_N= g_N'+g_{N+1}'$, where $g_N', g_{N+1}'\in {\Bbb L}_+({\cal X})-\{0\}$.
Therefore we get a new decomposition $f=g_1+\ldots +g_N'+g_{N+1}'$, which contradicts the assumption that $N$ is the maximal number.

Let us combine the same terms appearing in the right hand side of (\ref{last}). It gives us a decomposition
\be \la{9.9.1.1}
f=\sum_{f_i\in {\bf E}({\cal X})} c_i f_i, ~~~~c_i\in \Z_{\geq 0}. 
\ee
 If ${\bf E}({\cal X})$ is a $\Z$-basis of the coordinate ring, then the decomposition (\ref{9.9.1.1}) is unique. Conversely, if $f$ can be decomposed as in (\ref{9.9.1.1}), then $f\in {\Bbb L}_+({\cal X})$. The Lemma is proved.
\end{proof}

\bd \la{eq20}
The support of $f$ is a finite set 
\be \la{eq22} 
S_f:=\{l\in {\cal A}_{A_n}(\Z^t)~|~ c(f;l)\neq 0\}.
\ee
\ed

{\bf Example.} We consider the case of type $A_2$. Let
\be
l_1=(-1,0),~l_2=(0,1),~l_3=(1,1),~l_4=(1,0),~l_5=(0,-1).
\ee
In a certain coordinate system, the set
\be
{\cal A}_{A_2}(\Z^t)\simeq \Z^2=\{bl_i+cl_{i+1}~|~b,c\in \Z_{\geq 0},~i\in \Z/5\}.
\ee
We have
\begin{align}
{\Bbb I}_{\cal A}(l_1)=X_1^{-1},~~~{\Bbb I}_{\cal A}(l_2)=X_2,~~~{\Bbb I}_{\cal A}(l_3)=X_1X_2+X_2,\nonumber\\
{\Bbb I}_{\cal A}(l_4)=X_1+X_1X_2^{-1}+X_2^{-1},~~~{\Bbb I}_{\cal A}(l_5)=X_2^{-1}+X_2^{-1}X_1^{-1}.
\end{align}
In general
\be
{\Bbb I}_{\cal A}(bl_i+cl_{i+1})=\big({\Bbb I}_{\cal A}(l_i)\big)^{b}\cdot\big({\Bbb I}_{\cal A}(l_{i+1})\big)^c.
\ee 
Every $f\in {\bf P}({\cal X}_{A_2})$ can be expressed as
\be
f=\prod_{i=1}^5\big({\Bbb I}_{\cal A}(l_i)\big)^{d_i}=\sum_{l\in {\cal A}_{A_2}(\Z^t)}c(f;l){\Bbb I}_{\cal A}(l).
\ee
Here $d_1,\ldots,d_5\in \Z_{\geq 0}$. The coefficients are
\be
c(f;bl_i+cl_{i+1})=\sum_k{d_{i+3}\choose k}{d_{i+4}+k \choose d_{i+2}+d_{i+3}-d_i+b}{d_{i+2}\choose d_{i+4}-d_{i+1}+c+k}. 
\ee
For example, if $d_1=10, ~d_2=30,~d_3=10,~d_4=20,~d_5=30$, the support $S_f$ is shown in Fig.\ref{fig1}.

\subsection{Stasheff polytopes}
\label{sec: 2} \label{sec: 2.1}\label{sec: 2.2}
 We define Stasheff polytopes in ${\cal A}_{\Phi}({\Bbb A}^t)$ by using tropical cluster variables. We borrow the following notations from \cite{FZ2}.
 
Recall the set $\Phi_{\geq -1}$ of almost positive roots (i.e. positive and simple negative roots). There is a \emph{compatibility~degree~map} ([$loc. cit$]):
\be \la{eqq}
\Phi_{\geq -1} \times \Phi_{\geq -1}\lra \mathbb{Z}_{\geq 0},~~~
(\alpha, \beta)\lms(\alpha||\beta).
\ee
In particular, $(\alpha||\beta)=0$ implies $(\beta||\alpha)=0$. In this case, $\alpha$ and $\beta$ are called $compatible$. A $compatible$ $set$ $T$ is a subset of $\Phi_{\geq -1}$ whose elements are mutually compatible. We set
\be \la{eq27}
{\rm S}(T)=\{\alpha \in \Phi_{\geq -1}~|~ \alpha \notin T \text{ and }T\cup \{\alpha\} \text{ is still a compatible set}\}
\ee
and call ${\rm S}(T)$ the \emph{supplement} of $T$. A compatible set is called the \emph{cluster} associated to $\Phi$ if its supplement is empty. From [$loc.cit$, Theorem 1.8], each cluster is a $\mathbb{Z}$-basis of the root lattice. 
 \vskip 3mm

Recall the cluster ${\cal A}$-variety ${\cal A}_{\Phi}$ of classical type. Its cluster variables $A_{\alpha}$ are indecomposable functions indexed by ${\Phi_{\geq -1}}$ ([FZ1, Theorem 4.27, 5.7]). Let $c=\{c_{\alpha}\}$ be a set of real numbers.
For each compatible set \emph{T}, define the $T$-face
\be \la{eq28}
\mc{F}_c^T=\{x\in \tropspr{A}~|~ A_{\alpha}^t(x)=c_{\alpha}\text{ for all } {\alpha}\in T, \text{ and } A_{\alpha}^t(x)\leq c_{\alpha} \text{ for all } \alpha\in {\rm S}(T)\}.
\ee 
If $T$ is empty, then by definition
\begin{equation}
\mc{F}_c^\emptyset=\{x\in \tropspr{A}~|~ A_{\alpha}^t(x)\leq c_{\alpha} \text{ for all } \alpha \in \Phi_{\geq-1}\}.
\label{sta}
\end{equation}
\bd \la{eq32}
The polytope $\mc{F}_{c}^{\emptyset}$ defined in \eqref{sta} is called a Stasheff polytope {\rm (}or a $generalized$ $associahedron${\rm )} if
\be \la{eq31} 
 T_1\subseteq T_2 \Longrightarrow\mc{F}_c^{T_2}\subseteq \mc{F}_c^{T_1},~~\forall ~T_1, T_2.
\ee
\ed

{\bf Remark.} 
 For simplicity, we will only consider the cases when the seed is reduced (see Section \ref{sec: 4.1}). For each $T$, there is a coordinate system in which the tropical cluster variables of \emph{T} are simultaneously linear. Then ${\cal F}_c^T$ becomes a convex polytope and thus homeomophic to a unit ball $D^k$. Clearly its dimension $k\leq n-\#T$, where \emph{n} is the rank of $\Phi$.
We say ${\cal F}_c^\emptyset$ is non-degenerate if $k=n-\#T$ for each \emph{T}. 
  
  The set $\mc{A}_{\Phi}(\mathbb{A}^t)$ is a subset of $\mc{A}_{\Phi}(\mathbb{R}^t)$. With an abuse of notation, the intersection $\mc{A}_{\Phi}(\mathbb{A}^t)\bigcap {\cal F}_{c}^{\emptyset}$ is called a Stasheff polytope in $\mc{A}_{\Phi}(\mathbb{A}^t)$.
  
  \vskip 3mm
  
 {\bf Examples}. We consider the case of type $A_2$. Then $\Phi_{\geq -1}=\{-\alpha, -\beta, \alpha,\beta, \alpha+\beta\}$. The cluster variables are
 \begin{equation}
 \begin{array}{cc}
 A_{-\alpha}=A_1, \quad A_{-\beta}=A_2, \quad A_{\alpha}=A_1^{-1}+A_1^{-1}A_2, \\ 
 A_{\beta}=A_2^{-1}+A_2^{-1}A_1,\quad A_{\alpha+\beta}=A_1^{-1}A_2^{-1}+A_1^{-1}+A_2^{-1}.
 \end{array}
 \end{equation} 
 Their tropicalization and a set $c=\{c_i\}$ of real numbers are given as follows
\begin{equation} \quad \left\{ \begin{array}{llllll}  
A_{-\alpha}^t&=&a_1, &\quad c_{-\alpha}&=&20;\\
A_{-\beta}^t&=&a_2,  &\quad c_{-\beta}&=&10;\\
A_{\alpha}^t&=&\max\{-a_1, -a_1+a_2\}, &\quad c_{\alpha}&=&20;\\
A_{\beta}^t&=&\max\{-a_2, a_1-a_2\}, &\quad c_{\beta}&=&20;\\
A_{\alpha+\beta}^t&=&\max\{-a_1-a_2, -a_1, -a_2\}, &\quad c_{\alpha+\beta}&=&30.
\end{array} \right. \end{equation} 
Fig.\ref{fig1} shows the Stasheff polytope ${\cal F}_{c}^{\emptyset}$ in $\mc{A}_{A_2}(\mathbb{R}^t)$. Fig.\ref{fig2}-\ref{fig5} show the same Stasheff polytope in $\mc{A}_{A_2}(\mathbb{R}^t)$ in the other four coordinate systems.

\begin{figure}[ht]
\epsfxsize250pt
\centerline{\epsfbox{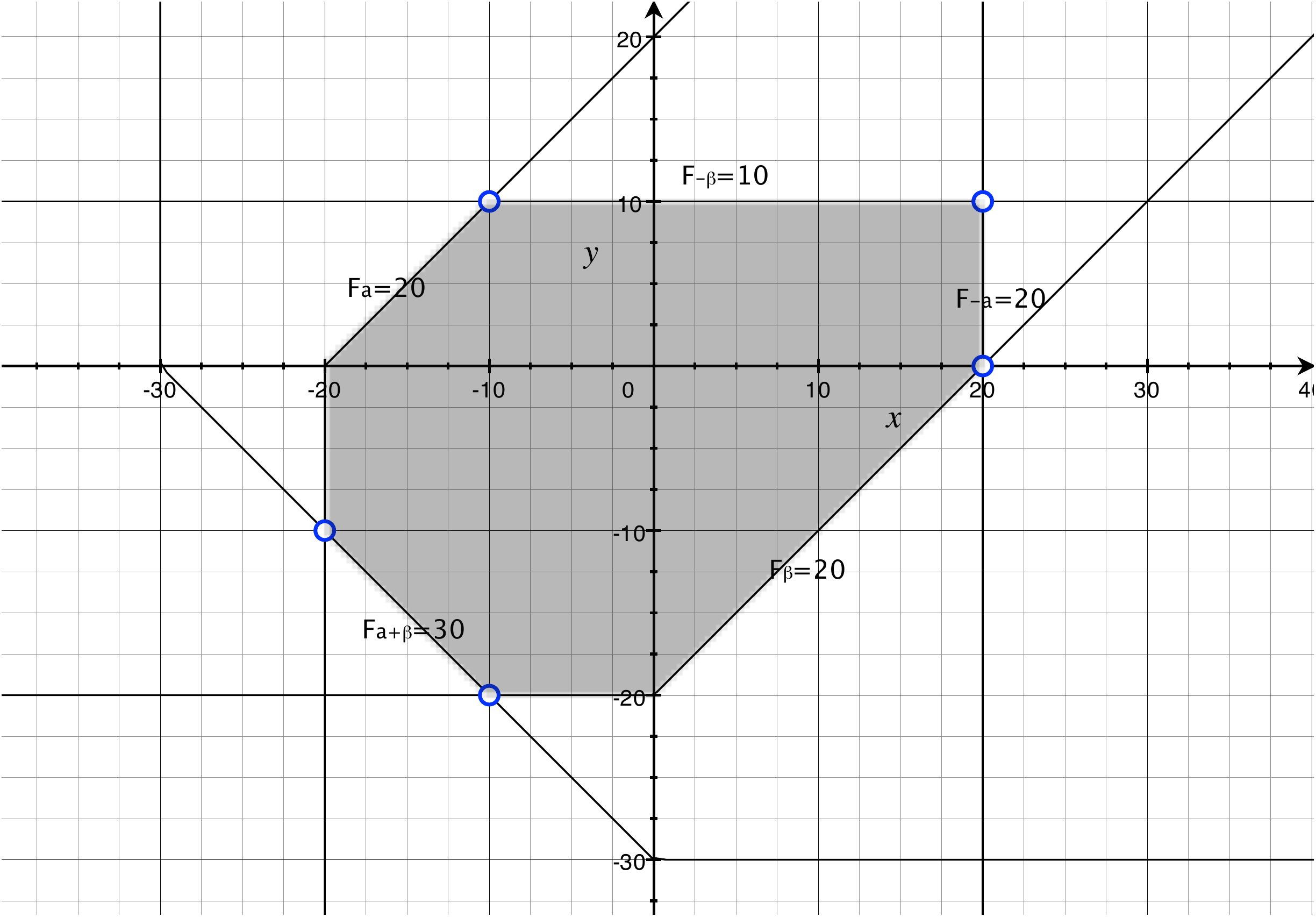}}
\caption{A Stasheff polytope in the tropical cluster $\mc{A}$-variety of type $A_2$.}
\label{fig1}
\end{figure}
\begin{figure}[ht]
\epsfxsize250pt
\centerline{\epsfbox{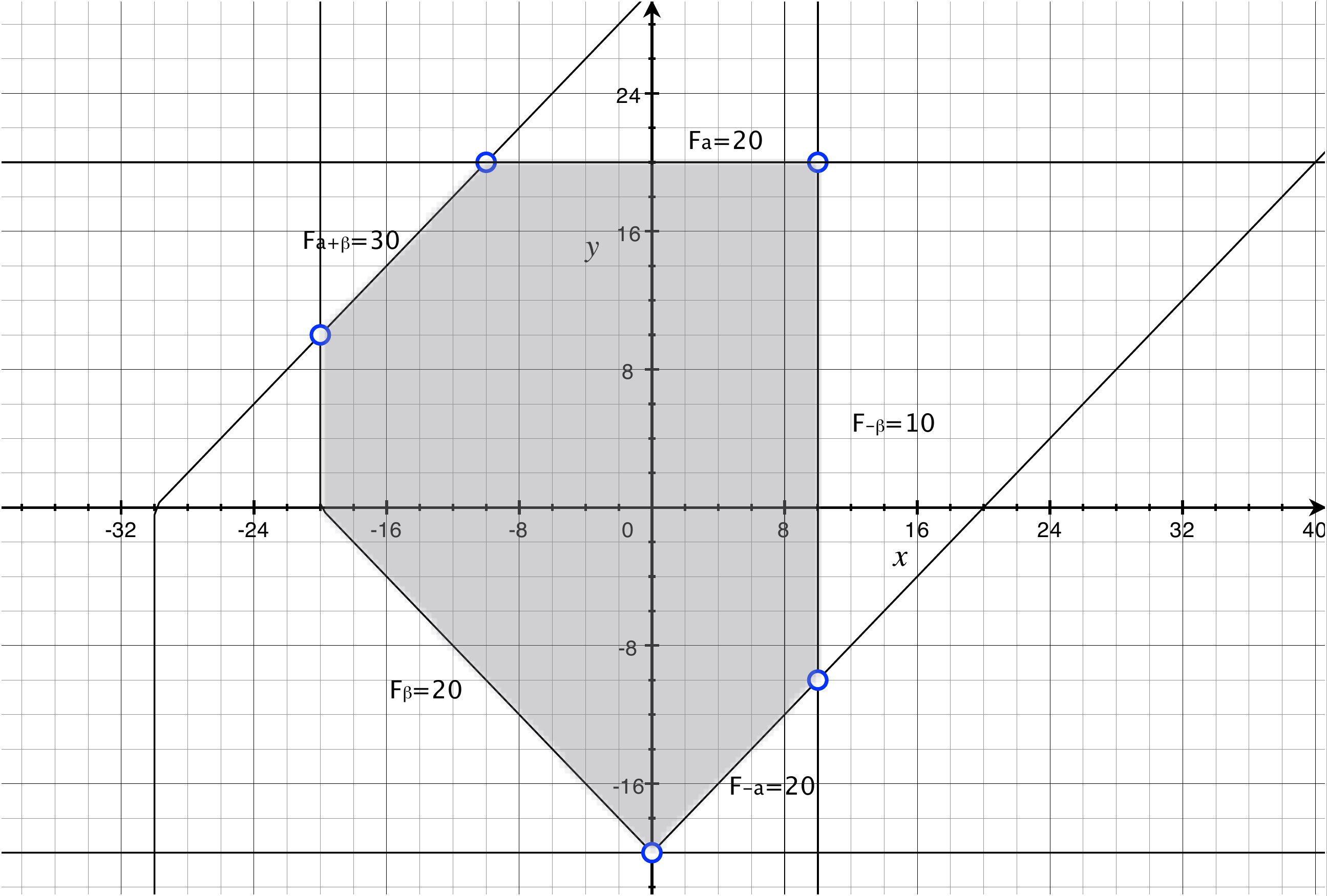}}
\caption{A Stasheff polytope in the tropical cluster $\mc{A}$-variety of type $A_2$.}
\label{fig2}
\end{figure}
\begin{figure}[ht]
\epsfxsize250pt
\centerline{\epsfbox{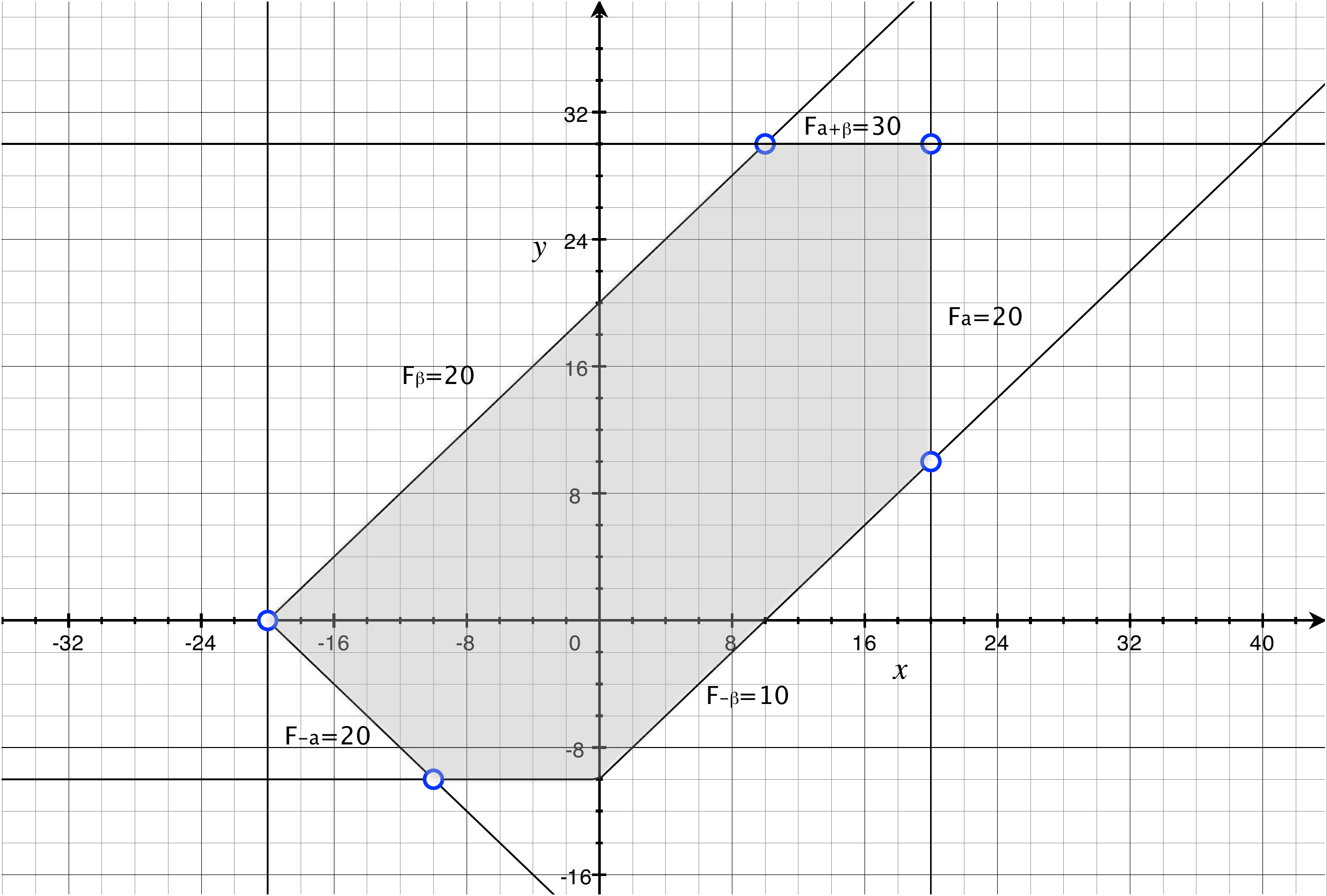}}
\caption{A Stasheff polytope in the tropical cluster $\mc{A}$-variety of type $A_2$.}
\label{fig3}
\end{figure}
\begin{figure}[ht]
\epsfxsize250pt
\centerline{\epsfbox{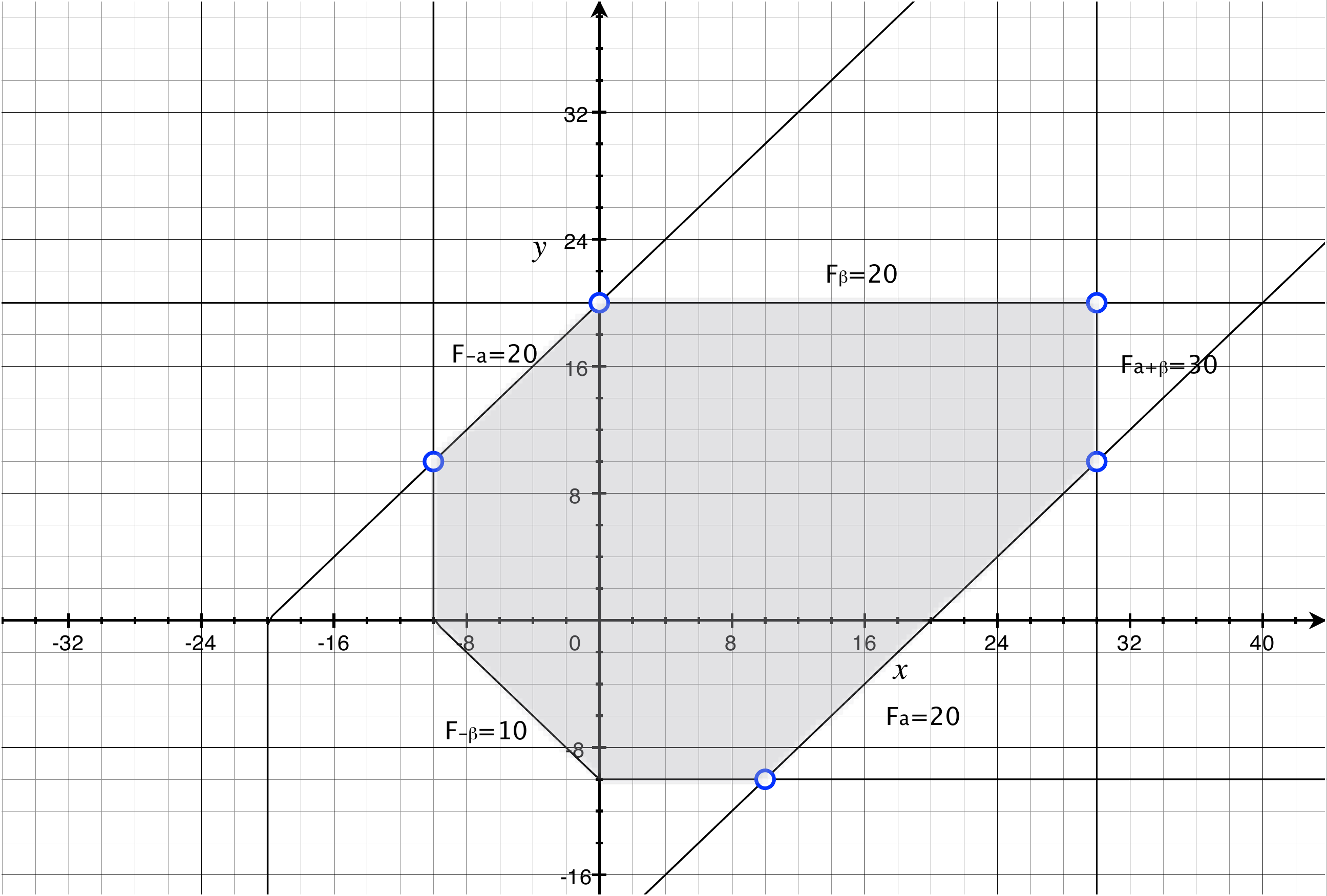}}
\caption{A Stasheff polytope in the tropical cluster $\mc{A}$-variety of type $A_2$.}
\label{fig4}
\end{figure}
\begin{figure}[ht]
\epsfxsize250pt
\centerline{\epsfbox{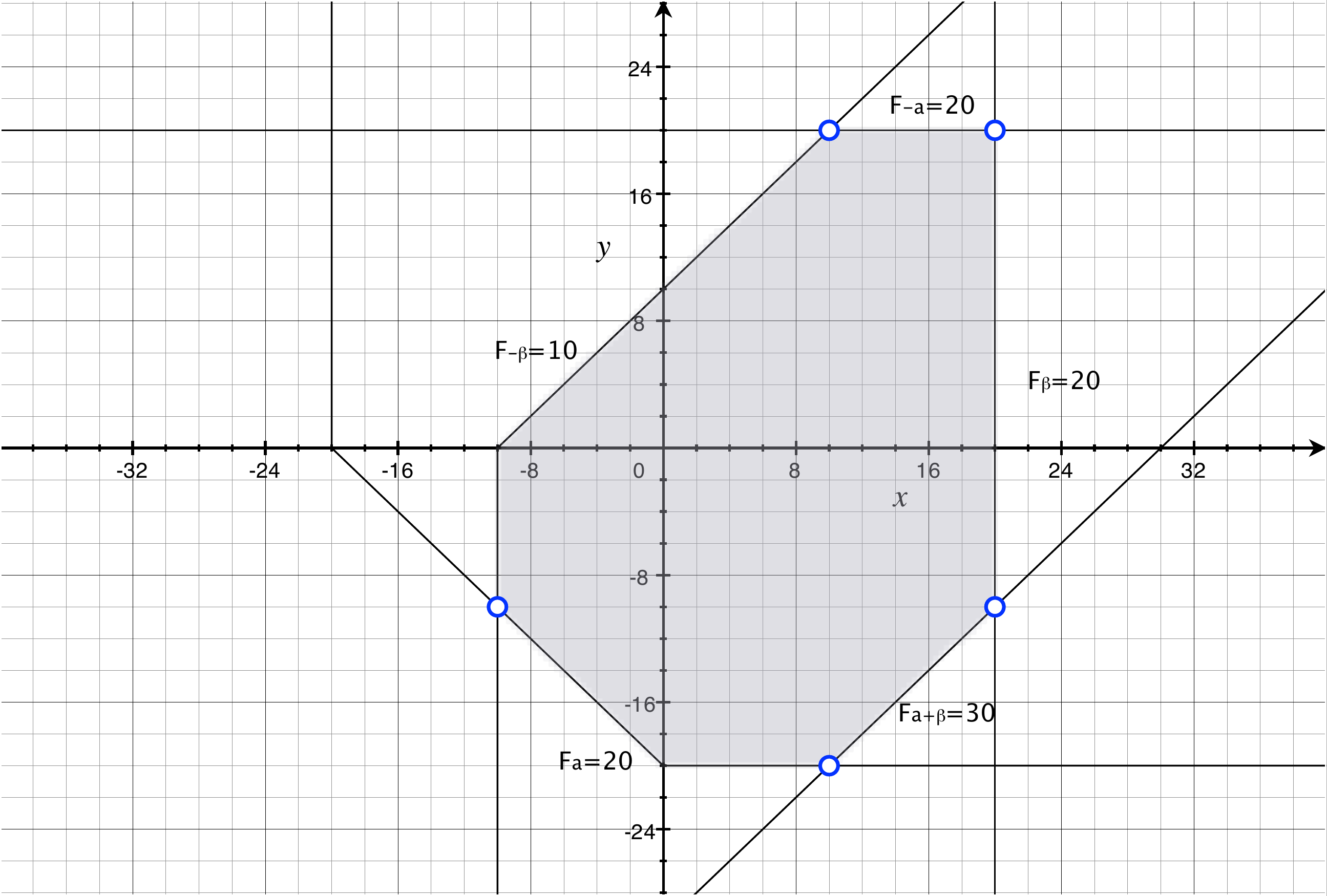}}
\caption{A Stasheff polytope in the tropical cluster $\mc{A}$-variety of type $A_2$.}
\label{fig5}
\end{figure}
\vskip 2mm

Fig.\ref{fig6} is a generalized associahedron of type $B_2$. 

\vskip 2mm

Fig.\ref{fig7} shows a Stasheff polytope of type $A_3$. There are three kinds of lines: visible lines which form part of the faces, lines on the back of the polytope which form part of faces (which are dashed), and ``creases" (which are dotted). Although the three creases are part of the boundary, they are not faces of the Stasheff polytope in the sense of this paper. They appear because the defining functions $A_{\alpha}^t\leq c_{\alpha}$ are not linear in general.  The creases on ${\cal F}_c^T$ will disappear once the tropical cluster variables of \emph{T} are simultaneously linear. For the same reason, the white vertex is not a face either. 

\vskip 2mm

Let $\Pi=\{-\alpha_1,\ldots, -\alpha_n\}$ be the set of simple negative roots. It is a cluster associated to $\Phi_{\geq -1}$. The set $\{A_{-\alpha_1},\ldots, A_{-\alpha_n}\}$ is a positive coordinate system of ${\cal A}_{\Phi}$.
For any root $\alpha=c_1\alpha_1+\ldots+c_n\alpha_n\in {\Phi_{\geq -1}}$, [FZ1, Theorem 5.8] shows that the cluster variable
\be
A_{\alpha}=\frac{P_{\alpha}(A_{-\alpha_1},\ldots, A_{-\alpha_n})}{A_{-\alpha_1}^{c_1}\ldots A_{-\alpha_n}^{c_n}},
\ee
where $P_{\alpha}$ is a polynomial in $A_{-\alpha_1},\ldots, A_{-\alpha_n}$ with nonzero constant term. Clearly every $A_{\alpha}^t$ becomes linear in the negative part
\be \la{6.10.15.03}
\{x\in{\cal A}_{\Phi}({\R}^t)~|~A_{-\alpha_i}^t(x) \leq 0, ~i=1,\ldots,n\}\simeq {\R^n_{\leq 0}}.
\ee
If a Stasheff polytope ${\cal F}_{c}^{\emptyset}$ is contained in (\ref{6.10.15.03}), then all faces of ${\cal F}_{c}^{\emptyset}$ become flat. It coincides with the usual Stasheff polytope in $\R^n$. When ${\Phi}$ is of type $A_n$, such a Stasheff polytope can be easily constructed. For example, given any Stasheff polytope  ${\cal F}_{c}^{\emptyset}$, one can choose a point $x=(-x_1,\ldots, -x_n)\in$ (\ref{6.10.15.03}) with  $x_i\geq 0$ large enough such that the Minkowski sum $\{x\}+{\cal F}_{c}^{\emptyset}$ is contained in (\ref{6.10.15.03}). We can show that $\{x\}+{\cal F}_{c}^{\emptyset}$ is still a 
 Stasheff polytope. It looks like the usual Stasheff polytope in $\R^n$. Fig.\ref{fig9} shows an example of such a polytope.
 
\begin{figure}[ht]
\epsfxsize350pt
\centerline{\epsfbox{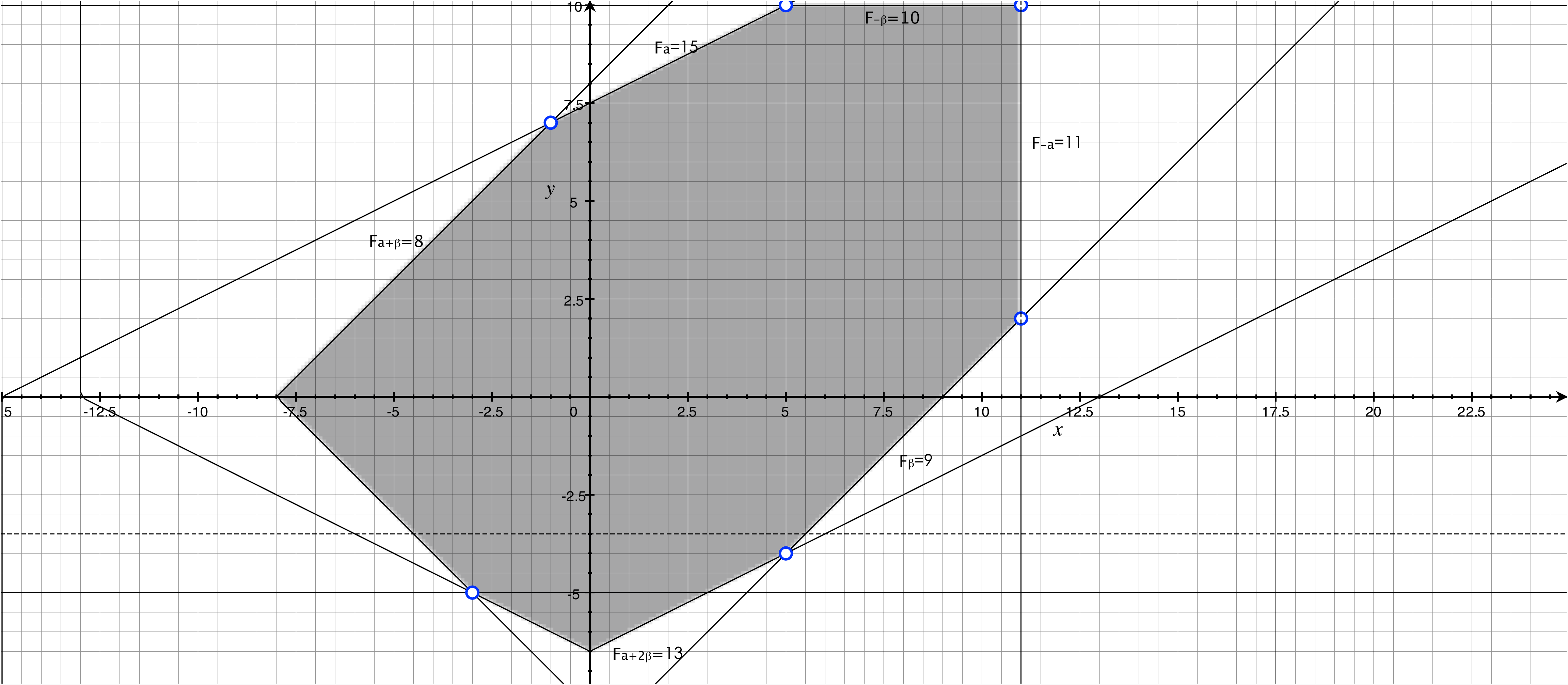}}
\caption{A generalized associahedron of type $B_2$.}
\label{fig6}
\end{figure}

\begin{figure}
\epsfxsize160pt
\centerline{\epsfbox{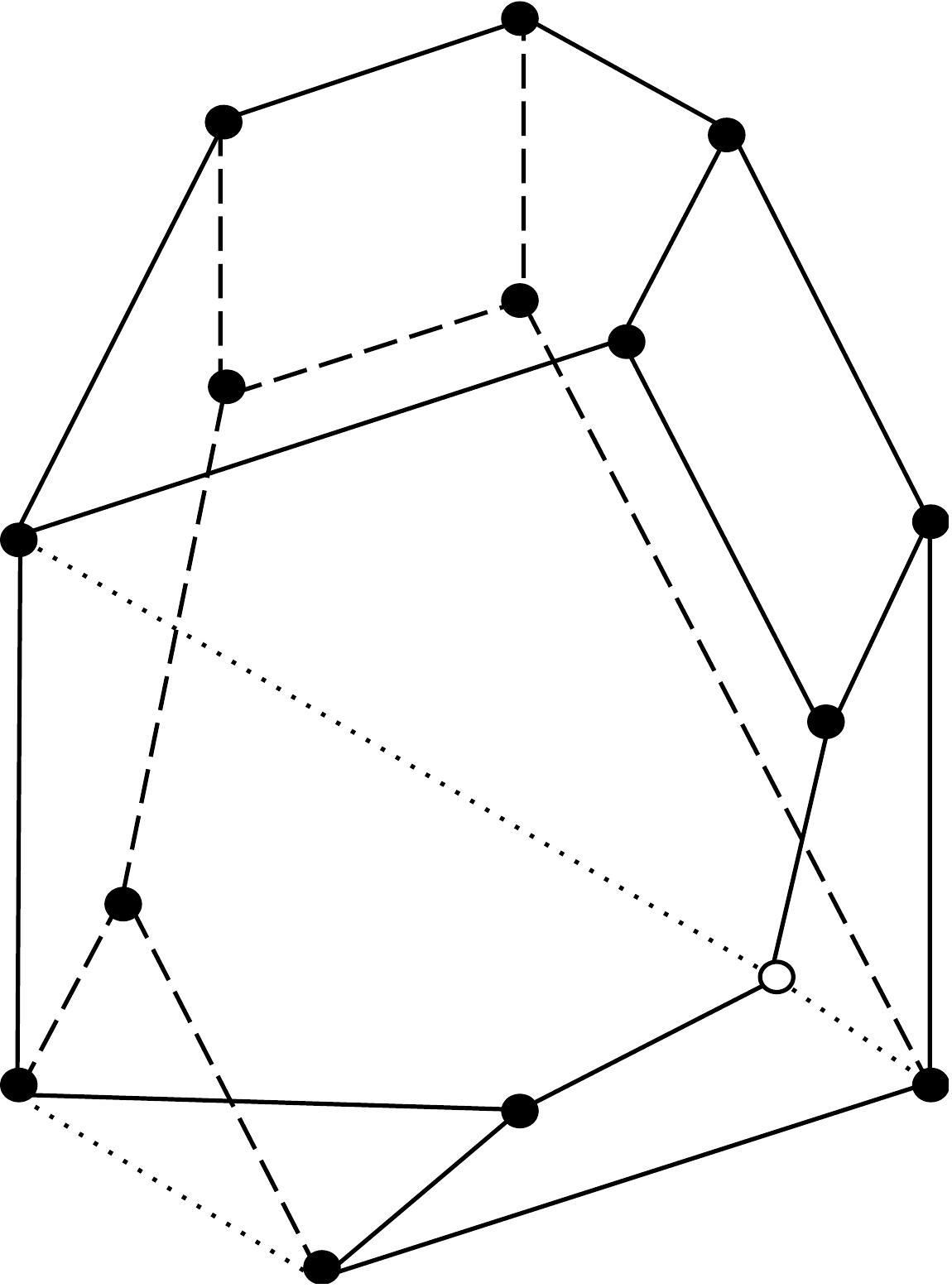}}
\caption{A Stasheff polytope of type $A_3$.}
\label{fig7}
\end{figure}

\begin{figure}[ht]
\epsfxsize250pt
\centerline{\epsfbox{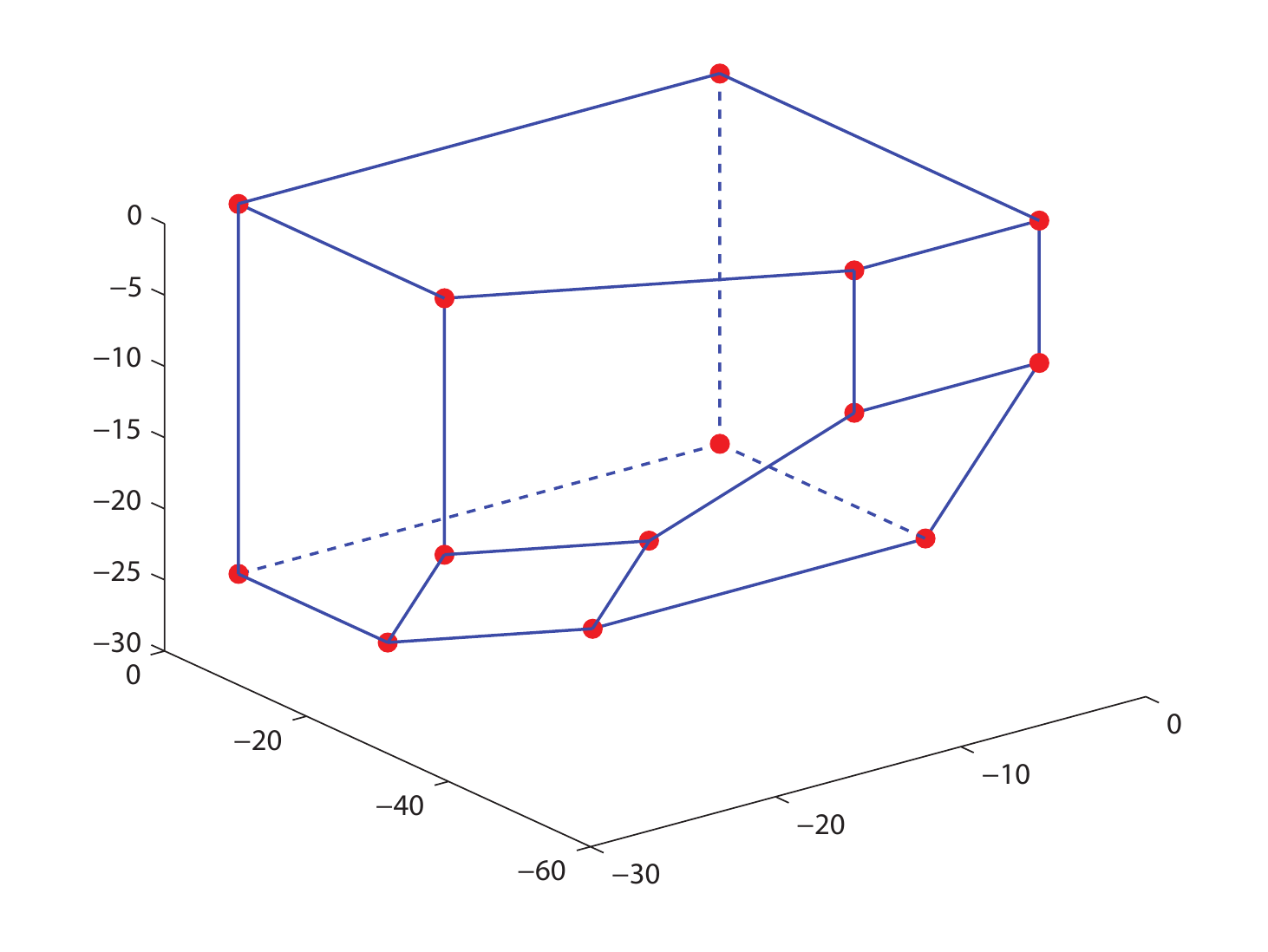}}
\caption{A Stasheff polytope of type $A_3$.}
\label{fig9}
\end{figure}

\vskip 2mm

We will focus on the cluster $\mc{A}$-variety $\mc{A}_{A_n}$ (see Definition \ref{eq37}). Our basic result is as follows. 
\bt
\la{6.3.12.44}
Every single point set $\{l\} \subset \mc{A}_{A_n}(\mathbb{A}^t)$ is a convex set. Given finitely many points $l_1, ..., l_m \in \mc{A}_{A_n}(\mathbb{A}^t)$, their Minkowski sum
\be\sum_{k=1}^ml_k:=\{x\in \mc{A}_{A_n}(\mathbb{A}^t)~|~ F^t(x)\leq \sum_{i=1}^mF^t(l_k) \text{ for all } F\in {\bf E}(\mc{A}_{A_n})\}
\ee
is a Stasheff polytope ${\cal F}_{c}^{\emptyset}$ with
\be
c_{\alpha}=\sum_{k=1}^m A_{\alpha}^t(l_k),~~~\forall \alpha\in \Phi_{\geq -1}.
\ee
\et


\section{Cluster $\mc{A}$-varieties of type $A$}
\label{sec: 3}

\subsection{Definitions and a criterion for Stasheff polytopes}
\label{sec: 3.1}
In this section, we give two realizations of cluster $\mc{A}$-varieties of type $A_n$ (Definition \ref{eq35}, \ref{eq37}), which are the models used in the proof of Theorem \ref{6.3.12.44}. 

Let \emph{S} be an (\emph{n}+3)-gon with vertices labeled by 1 through (\emph{n}+3) clockwise.  We consider the line segments connecting two different vertices. Segments connecting two adjacent vertices are called \emph{edges}. Segments that are not edges are called \emph{diagonals}. 

With an abuse of notation, we will identify a triangulation $T$ of the (\emph{n}+3)-gon with the set of diagonals contained in $T$. Denote by $\widetilde{T}$ the union of $T$ and the set of edges. 

Recall the notion of compatible sets and clusters from Section \ref{sec: 2.2}. In the case of type $A_n$, they have the following description.

\bp [\cite{FZ2}]\la{eq33}
 Let $\Phi$ be a root system of type $A_n$. Its almost positive roots are one-to-one corresponding to the diagonals of an {\rm (\emph{n}+3)}-gon.  Two almost positive roots are compatible if and only if they correspond to two diagonals which have no intersection inside. Every compatible subset  thus corresponds to a {\rm (\emph{partial})} triangulation of the {\rm (\emph{n}+3)}-gon. It is a cluster if and only if the triangulation is complete.
\ep

\vskip 2mm

Denote by $\{ij\}$ the segment connecting the vertices labeled by $i,j$. Assign to each segment $\{ij\}$ a variable $A_{ij}$.  When $i,j,k,l$ are seated clockwise, the following is called the Pl\"{u}cker relation:
\be \la{eq34}
A_{ik}A_{jl}=A_{ij}A_{kl}+A_{il}A_{jk}.
\ee

\bd \la{eq35}
Assign to each complete triangulation T a coordinate system
\be \la{eq36}
\widetilde{\alpha}_{T}=\{A_{ij}~|~\{ij\} \in \widetilde{T}\}.
\ee
Each $\widetilde{\alpha}_{T}$ determines a torus ${\rm Spec}(\mathbb{C}[A_{ij}^{\pm}]),~\{ij\}\in\widetilde{T}$. Denote by $\widetilde{\mc{A}}_{A_n}$ the space obtained by gluing the tori via the transition maps generated by the Pl\"{u}cker relation \eqref{eq34}. 
\ed
\bd \la{eq37}
Assign to each complete triangulation T a coordinate system
 \be \la{eq38}
 \alpha_{T}=\{A_{ij}~|~ \{ij\} \in T\}. 
 \ee
Each $\alpha_{T}$ determines a torus ${\rm Spec}(\mathbb{C}[A_{ij}^{\pm}]),~\{ij\}\in T$. Denote by $\mc{A}_{A_n}$ the space obtained by gluing the tori via the transition maps generated by the Pl$\ddot{u}$cker relation \eqref{eq34} and the following condition:
\be
\la{eq39}
A_{ij}=1 \text{ for all edges } \{i j\}.
\ee
\ed
${\bf Remark.}$ The transition maps defined above are subtraction free. Thus both $\widetilde{\mc{A}}_{A_n}$ and $\mc{A}_{A_n}$ are positive spaces. They are cluster $\mc{A}$-varieties of type $A_n$ with different coefficients. By \cite{FZ1} \cite{C}, indecomposable functions on these spaces are monomials in the $A_{ij}$'s from the same coordinate systems.

\vskip 2mm

 In particular, we are interested in $\mc{A}_{A_n}$ which corresponds to the case with reduced seed. Given  a set of real numbers  $c=\{c_{ij}\}$ indexed by the set of diagonals, every partial triangulation $T$ gives rise to a face:
\begin{align} \la{eq40}
\mc{F}_c^T &=\{x\in \mc{A}_{A_n}(\mathbb{R}^t)~|~A_{ij}^t(x)=c_{ij} \text{ for all diagonals } \{ij\}\in T,  \nonumber\\ &\qquad \text{ and } A_{ij}^t(x)\leq c_{ij} \text{ for all diagonals } \{ ij\}\in {\rm S}(T)\}.
\end{align}
When $T$ contains no diagonals, we obtain the $\emptyset$-face: 
\be \la{eq41}
\mc{F}_c^{\emptyset}=\{x\in \mc{A}_{A_n}(\mathbb{R}^t)~|~ A_{ij}^t(x)\leq c_{ij} \text{ for all diagonals } \{ij\}\}.
\ee
If $T$ is complete, then  $\mc{F}_c^T$ contains only one point. In this case, $\mc{F}_c^T$ is called a vertex. Our next theorem provides a criterion for recognizing Stasheff polytopes in $\mc{A}_{A_n}(\mathbb{R}^t)$.\label{sec: 3.2}
\bt \la{eq42}
The polytope ${\cal F}_{c}^{\emptyset}$ defined by \eqref{eq41} is a Stasheff polytope if and only if every vertex $\mc{F}_c^T$ is contained in ${\cal F}_{c}^{\emptyset}$.
\et

\begin{proof}
The ``only if" part is by definition. 

We tropicalize the Pl\"{u}cker relation and the condition \eqref{eq39}:
\be \la{eq44}
A_{pr}^t+A_{qs}^t=\max\{A_{pq}^t+A_{rs}^t, A_{qr}^t+A_{ps}^t\}  \quad\text{ if $p,q,r,s$ are seated clockwise};
\ee
\begin{equation}
A_{pq}^t=0 \quad \text{ if $p,q$ are adjacent.}
\end{equation}

For the ``if" part, given a partial triangulation $T_1$, let $\alpha=\{ij\}\in {\rm S}(T_1)$. Then  $T_2=T_1\bigcup\{\alpha\}$ is still a triangulation. By induction, we only need to show that $\mc{F}_c^{T_2} \subseteq \mc{F}_c^{T_1}$.
It is enough to show that $\forall\{kl\}\in{\rm S}(T_1),~\forall y\in {\cal F}_c^{T_2},$
\be \la{6.10.16.34}
A_{kl}^t(y)\leq c_{kl}.
\ee

If $\{kl\}\in {\rm S}(T_2)$ or $\{kl\}=\{ij\}$, then (\ref{6.10.16.34}) follows by definition. Otherwise $\{kl\}$ intersects $\{ij\}$. Thus the vertices $i, k, j, l$ are seated clockwise. Consider the segments $\{ik\},\{kj\}, \{jl\}, \{il\}$. They are either in ${\rm S}(T_2)\cup T_1$ or they are edges. Meanwhile, these four segments are mutually compatible. Therefore, there exists a cluster $T$ that contains $T_2$ and all the diagonals among $\{ik\},\{kj\}, \{jl\}, \{il\}$. Such a cluster $T$ gives rise to a vertex $\mc{F}_c^{T}:=\{x\}.$ Now by assumption, one has $x\in {\cal F}_{c}^{\emptyset}$. Therefore
\be
\la{eq45}
c_{kl}\geq A_{kl}^t(x)=\max\{A_{ik}^t(x)+A_{jl}^t(x), A_{il}^t(x)+A_{kj}^t (x)\}-A_{ij}^t(x)=\max\{c_{ik}+c_{jl}, c_{il}+c_{kj}\}-c_{ij}.
\ee 
Here $c_{pq}=0$ if the segment $\{pq\}$ is an edge. For any $y\in \mc{F}_c^{T_2}$, by definition,
\be
\la{eq46}
A_{ik}^t(y)\leq c_{ik}, \quad A_{jl}^t(y)\leq c_{jl}, \quad A_{il}^t(y)\leq c_{il}, \quad A_{kj}^t (y)\leq c_{kj}, \quad A_{ij}^t(y)=c_{ij}.
\ee
By \eqref{eq44}, \eqref{eq45} and \eqref{eq46}, 
\be
\la{eq47}
A_{kl}^t(y)=\max\{A_{ik}^t(y)+A_{jl}^t(y), A_{il}^t(y)+A_{kj}^t (y)\}-A_{ij}^t(y)\leq c_{kl}.
\ee 
The Theorem is proved. 
\end{proof}

\subsection{Proof of Theorem \ref{6.3.12.44}}
\label{sec: 3.3}
Theorem \ref{6.3.12.44} is a consequence of the following Lemmas. 
\bl \la{eq48}
Given finitely many points $l_1, ..., l_m \in \mc{A}_{A_n}(\mathbb{R}^t)$ and a set $c=\{c_{ij}\}$ of real numbers:
\be
\la{eq49}
c_{ij}=\sum_{k=1}^mA_{ij}^t(l_k),
\ee
the polytope ${\cal F}_{c}^{\emptyset}$ defined by \eqref{eq41} is a Stasheff polytope.
\el
\begin{proof}
By Theorem \ref{eq42}, it is enough to show that all vertices are contained in ${\cal F}_{c}^{\emptyset}$. Each cluster $T$ gives rise to a coordinate system $\alpha_{T}=\{A_{ij}~|~ \{ij\}\in T\}$, that maps $\mc{A}_{A_n}(\mathbb{R}^t)$ isomorphically to $\mathbb{R}^n$. Let $l$ be the sum of $l_k$'s as vectors in $\mathbb{R}^n$. Clearly $A_{ij}^t(l)=c_{ij} ,~\forall\{ij\}\in T$. In other words, $\mc{F}_c^T=\{l\}$. Meanwhile, all tropical cluster variables are convex piecewise linear functions in this coordinate system $\alpha_T$.  The convexity shows that 
\be
\la{eq51}
A_{ij}^t(l)\leq\sum_{k=1}^mA_{ij}^t(l_k)=c_{ij}, \text{ for all diagonals } \{ij\}.
\ee
Therefore $\mc{F}_{c}^T\subseteq {\cal F}_{c}^{\emptyset}$. The Lemma is proved.
\end{proof}
\bl
\la{eq52}
The polytope ${\cal F}_{c}^{\emptyset}$ in Lemma \ref{eq48}  is the Minkowski sum of the defining $l_i$'s. 
\el
\begin{proof}
By definition, the Minkowski sum of the defining $l_i$'s is contained in ${\cal F}_{c}^{\emptyset}$. It remains to show that ${\cal F}_{c}^{\emptyset}$ is contained in the Minkowski sum.
For cluster $\mc{A}$-varieties of classical type, all indecomposable functions are cluster monomials. Namely,  $\forall F\in \exset{A}$, there exists a cluster $T$ such that  $F^t=\sum_{\{ij\}\in T} a_{ij}A_{ij}^t$, where the numbers $a_{ij}$ are all nonnegative integers. Thus $\forall x\in {\cal F}_{c}^{\emptyset}$,
\be
\la{eq54}
F^t(x)=\sum_{\{ij\}\in T}a_{ij}A_{ij}^t(x)\leq\sum_{\{ij\}\in T}a_{ij}c_{ij}=\sum_{k}F^t(l_k).
\ee
 The Lemma is proved.
\end{proof}

\begin{proof} (Of Theorem \ref{6.3.12.44})
By Lemma \ref{eq48} and \ref{eq52}, the second part is proved. It remains to show that every single point set $\{l\}\subset\mc{A}_{A_n}(\mathbb{R}^t)$ is convex. Notice that its convex hull $C(\{l\})$ gives rise to a Stasheff polytope whose only vertex is $l$. By induction on the dimensions of the faces, one can easily show that all faces of this polytope contain only $l$. The first part is proved.
\end{proof}
{\bf Remark.} In fact, suppose $\{l\}$ is not convex, by the same argument used in the proof of Theorem  
\ref{eq61}, one can show that there exist cyclic ordered $k,s,m,t$ such that the tropical Pl\"ucker relation fails on $l$. This will give another proof of the convexity of $\{l\}$.

\vskip 3mm

The proof of Theorem \ref{eq42} crucially uses \eqref{eq45}. It provides another criterion for recognizing Stasheff polytopes.
\bt
\la{eq55}
The polytope ${\cal F}_{c}^{\emptyset}$ defined in \eqref{eq41} is a Stasheff polytope if and only if 
\begin{equation} \la{6.10.9.36}
c_{ij}+c_{kl}\geq \max\{c_{ik}+c_{jl}, c_{kj}+c_{il}\} \text{ for all $i,k,j,l$ seated clockwise.}
\end {equation}
Furthermore, ${\cal F}_{c}^{\emptyset}$ is non-degenerate if and only if these inequalities are strict for all $i,k,j,l$ seated clockwise.
\et
\begin{proof}
Let $T$ be a cluster containing all diagonals among $\{ik\},\{jl\},\{kj\},\{il\},\{ij\}$. Let ${\cal F}_{c}^T=\{x\}$. If $\{x\}\in {\cal F}_c^{\emptyset}$, then by (\ref{eq45}), the condition (\ref{6.10.9.36}) follows.  The ``only if'' part of the first statement is proved. The ``if'' part follows from the same argument used in the proof of Theorem \ref{eq42}.
By induction on the dimensions of the faces, the second statement follows.
\end{proof}

\subsection{Products of elements of a canonical basis}
\label{sec: 3.4}

Let $\widetilde{\mc{A}}_{A_n}$be as in Definition \ref{eq35}. The set ${\bf E}(\widetilde{\mc{A}}_{A_n})$ provides a canonical basis for the coordinate ring of $\widetilde{\mc{A}}_{A_n}$ (\cite{C}, Theorem 1.1). In this section, we study the partial order structure on ${\bf P}(\widetilde{\mc{A}}_{A_n})$.

Label the vertices of an (\emph{n}+3)-gon by 1 through (\emph{n}+3) clockwise as before. Every $F\in {\bf P}(\widetilde{\mc{A}}_{A_n})$ can be expressed as a product of the variables $A_{ij}$ corresponding to segments $\{ij\}$ of the (\emph{n}+3)-gon. Thus \emph{F} can be represented by a \emph{weighted~graph}.

\bd
\la{eq56} 
A weighted graph is a collection of segments of an {\rm (\emph{n}+3)}-gon with integral weights such that the weights of diagonals are non-negative. 
\ed

We present a weighted graph by a symmetric {\rm (\emph{n}+3)$\times$(\emph{n}+3)} matrix $G=(w_{ij})$, such that $w_{ij}$ is the weight of $\{ij\}$ if $i\neq j$ and $w_{ii}=0$ otherwise. The matrix $G$ is called trivial if all its entries are zero. 

The map 
\be \la{eq58}
\mathbb{I}(G)=\prod_{1\leq i< j\leq n+3}A_{ij}^{w_{ij}}
\ee 
induces a surjection from the set of weighted graphs to  ${\bf P}(\widetilde{\mc{A}}_{A_n})$. We provide a criterion for determining the partial order on ${\bf P}(\widetilde{\mc{A}}_{A_n})$.
\bd \la{eq59}
For any matrix $G=(w_{ij})$,  we set
\begin{align}
\la{eq60}
\Gamma_{kl}(G):=\frac{1}{2}\sum_{k\leq i,j \leq l}w_{ij},~~~\forall 1\leq k\leq l\leq n+3;\nonumber\\
R_{p}(G):=\sum_{1\leq j\leq n+3}w_{pj},~~~\forall 1\leq p \leq n+3 .
\end{align}
\ed
 \bt
\la{eq61}
 Given two weighted graphs presented by matrices $G_1$, $G_2$,  if 
 \begin{enumerate}
\item $ \Gamma_{kl}(G_1)\geq \Gamma_{kl}(G_2),~\forall 1\leq k\leq l\leq n+3$,
\item $ R_p(G_1)=R_p(G_2),~ \forall 1\leq p\leq n+3$,
 \end{enumerate}
  then $\mathbb{I}(G_1)\leq \mathbb{I}(G_2)$.
\et

 {\bf Remark.} Let $[p,q]:=\{p,\ldots, q\}$. For any $G=(w_{ij})$,
define 
\be
\la{eq63}
\disp{I_{kl}(G)=\sum_{k+1\leq i \leq l}R_i(G)-2 \Gamma_{k+1,l}(G)=\sum_{i\in I,j\in J}w_{ij}},
\ee
where $I=[k+1,l]$, $J=[1,n+3]-I$ are partitions of the set of vertices obtained by cutting the boundary of the (\emph{n}+3)-gon into two connected parts. Geometrically, $I_{kl}$ is the sum of weighted segments that connect $I$ and $J$, and $R_p(G)$ is the sum of weighted segments emanating from the vertex labeled $p$. The two conditions in Theorem \ref{eq61} is equivalent to the condition that
\begin{align}
\la{eq64}
R_p(G_1)=&R_p(G_2), ~\forall {~\rm vertices~}p,\nonumber\\
 I_{kl}(G_1)\leq &I_{kl}(G_2),~\forall {~\rm diagonals~} \{kl\}.
\end{align}
The condition \eqref{eq64} is more geometric and does not depend on the labeling of the vertices. Because of this, if necessary, one can relabel the vertices so that the cyclic order is preserved. The conditions in the Theorem are easier for calculation.

\vskip 2mm

Let $G_1=(u_{ij}), G_2=(v_{ij})$ be two weighted graphs satisfying condition \eqref{eq64}. Let $G=(s_{ij})$ such that 
 $s_{ij}=\min\{u_{ij},v_{ij}\}$. Then $G, G_1-G, G_2-G$ are all weighted graphs and 
 \be
 \la{eq65}
 \mathbb{I}(G_i)=\mathbb{I}(G_i-G)\mathbb{I}(G), \text{ for } i=1,2. 
 \ee
Clearly $\mathbb{I}(G_1-G)\leq\mathbb{I}(G_2-G)$ implies $\mathbb{I}(G_1)\leq \mathbb{I}(G_2)$. Notice that $G_1-G$, $G_2-G$ still satisfy condition \eqref{eq64}.  Theorem \ref{eq61} can be reduced to the case when 
 \be
 \la{eq66}
 G_1=(u_{ij}),~G_2=(v_{ij})\text{~such that~}
 \min\{u_{ij}, v_{ij}\}=0, ~\forall \{ij\}.
 \ee
\bd
\la{eq67}
The length of the segment $\gamma=\{i j\}$ is
\be
\la{eq68}
l(\gamma)=\min\{|i-j|, n+3-|i-j|\}.
\ee 
For each nontrivial weighted graph $G$, the depth of {G} is 
\be
{\rm dep}(G):=\min_{\gamma~|~w_{\gamma}\neq 0}\{l(\gamma)\}.
\ee
\ed 

Clearly both are well defined as long as the cyclic order of the labeling is preserved.
\bl
\la{eq69}
If $G_1$ and $G_2$ are two nontrivial weighted graphs satisfying both conditions \eqref{eq64} and \eqref{eq66}, then 
\be
\la{eq70}
1\leq {\rm dep}(G_1)<{\rm dep}(G_2)\leq \frac{n+3}{2}.
\ee
\el

\begin{proof} Here $1\leq {\rm dep}(G_1)$ and ${\rm dep}(G_2)\leq (n+3)/2$ are clear. It remains to show that  ${\rm dep}(G_1)<{\rm dep}(G_2)$. 

Let $\alpha$ be the shortest segment such that its weight $v_{\alpha}$ in $G_2$ is strictly positive. Let $k:=l(\alpha)={\rm dep}(G_2)$. Relabel the vertices such that $\alpha=\{1, k+1\}$. By definition and the first condition in Theorem \ref{eq61},  we get
\be \la{6.5.1}
0<v_{1,k+1}=\sum_{1\leq i<j\leq k+1}v_{ij}=\Gamma_{1, k+1}(G_2)\leq \Gamma_{1,k+1}(G_1)=\sum_{1\leq i<j\leq k+1}u_{ij}
\ee  
There is at least one nonzero $u_{ij}$ on the right hand side of (\ref{6.5.1}). Here $|i-j|< k$ unless $\{ij\}=\{1,k+1\}$. Since $v_{1,k+1}\neq 0$, condition \eqref{eq66} tells us that $u_{1,k+1}=0$. Therefore ${\rm dep}(G_1)<k={\rm dep}(G_2)$. The Lemma is proved.
\end{proof} 

\begin{proof}
(Of Theorem \ref{eq61})
Let $G_1=(u_{ij}), G_2=(v_{ij})$ be weighted graphs such that both \eqref{eq64} and \eqref{eq66} hold. If one of the graphs is trivial, then by the first condition of (\ref{eq64}), $R_l(G_1)=R_l(G_2)=0,~\forall 1\leq l\leq n+3$. By \eqref{eq66}, every weight is non-negative. Therefore both graphs are trivial. Theorem \ref{eq61} follows directly. 

Assume both $G_1$ and $G_2$ are nontrivial. Let $\alpha$ be the shortest segment such that its weight $u_{\alpha}$ in $G_1$ is nonzero. Let $k:=l(\alpha)={\rm dep}(G_1)$. Relabel the vertices such that $\alpha=\{k,n+3\}$. By Lemma \ref{eq69}, the depths of both $G_1$ and $G_2$ are greater than $k-1$. Thus $\Gamma_{1,k-1}(G_1)=\Gamma_{1,k-1}(G_2)=0$. By definition $\Gamma_{1,n+3}(G)=\frac{1}{2}\sum_{l=1}^{n+3}R_l(G)$. Thus $\Gamma_{1, n+3}(G_1)=\Gamma_{1,n+3}(G_2)$. Thus 
\be
\la{eq73}
\Gamma_{k,n+3}(G_1)=\Gamma_{1,n+3}(G_1)+\Gamma_{1, k-1}(G_1)-\sum_{1\leq i\leq k-1}R_i(G_1)=\Gamma_{k,n+3}(G_2).
\ee
The weight $u_{k, n+3}=u_{\alpha}>0$. Thus $R_k(G_2)=R_k(G_1)>0.$ By Lemma \ref{eq69}, ${\rm dep}(G_2)>k$. Therefore $v_{kj}=0$ for all $j \in [1 ,2k] \cup \{n+3\}$. There is at least one $j\in [2k+1,n+2]$ such that $v_{kj}>0$ because $R_{k}(G_2)>0$. Let $m$ be the largest one in $[2k+1, n+2]$ such that $v_{km}>0$. Notice that 
\begin{align}
\la{eq74}
\Gamma_{km}(G_2)+\Gamma_{m,n+3}(G_2)&\leq \Gamma_{km}(G_1)+\Gamma_{m,n+3}(G_1)\leq \Gamma_{k,n+3}(G_1)-u_{k,n+3}\nonumber \\ &<\Gamma_{k,n+3}(G_1)=\Gamma_{k,n+3}(G_2).
\end{align}
By the choice of $m$, we have $\disp{\sum_{m<t\leq n+3}v_{kt}=0.}$ Therefore, from \eqref{eq74}, 
\be
\la{eq75}
\sum_{k<s<m; \text{ }m<t\leq n+3}v_{st}=\sum_{k\leq s<m; \text{ }m<t\leq n+3}v_{st}=\Gamma_{k,n+3}(G_2)-\Gamma_{km}(G_2)-\Gamma_{m,n+3}(G_2)>0.
\ee

At least one of the $v_{st}$'s on the left hand side of \eqref{eq75} is strictly positive. Let $v_{st}>0$ be the one among them such that $s$ is the smallest, and $t$ is the smallest when $s$ is fixed. Now we have two positive entries $v_{km}$ and $v_{st}$, where the vertices labeled $k, s, m, t$ are seated clockwise. Define a new weighted graph $G_3=(w_{ij})$ via $G_2=(v_{ij})$ such that \eqarray{w_{ij}=w_{ji}}{ll}{v_{ij}-1,&\text{ if } \{ij\}=\{km\} \text{ or }\{st\};\\
                                              v_{ij}+1,&\text{ if } \{ij\}=\{kt\} \text{ or }\{sm\};\\
                                              v_{ij},     &\text{ otherwise} .
}
Therefore $ R_i(G_3)=R_i(G_2) \text{ for every vertex $i$},$ and 
\eqarray{\Gamma_{ij}(G_3)}{ll}{\Gamma_{ij}(G_2)+1 &\text{ if } k<i\leq s, \text{and } m\leq j<t\\
                                                            \Gamma_{ij}(G_2)      &\text{ otherwise.} \label{eqa}}

Clearly $\mathbb{I}(G_2)>\mathbb{I}(G_3)$ because of the Pl\"{u}cker relation $A_{km}A_{st}=A_{ks}A_{mt}+A_{kt}A_{ms}$.
If one can show that for these two new  weighted graphs $G_1, G_3$, condition \eqref{eq64} still holds, then Theorem \ref{eq61} is proved via induction. By \eqref{eqa}, it remains to show that
\be
\la{eq76}
\Gamma_{ij}(G_2)<\Gamma_{ij}(G_1) \text{ for all } k<i\leq s, \text{ and } m\leq j<t.
\ee
For all such $i,j$ in \eqref{eq76}
\Earray{c}{ \la{eq761}
\disp{\Gamma_{kj}(G_2)-\Gamma_{ij}(G_2)+\Gamma_{1,i-1}(G_2)-\Gamma_{1,k-1}(G_2)= \sum_{k\leq p <i ; 1\leq q\leq j} v_{pq}=\sum_{k\leq p <i}R_{p}(G_2)}\\
                   \disp{   =\sum_{k\leq p <i}R_{p}(G_1) >\sum_{k\leq p <i ; 1\leq q\leq j} u_{pq}=\Gamma_{kj}(G_1)-\Gamma_{ij}(G_1)+\Gamma_{1,i-1}(G_1)-\Gamma_{1,k-1}(G_1)}.
}
The inequality in (\ref{eq761}) follows from the fact that
$$\disp{ \sum_{k\leq p <i}R_{p}(G_1) -\sum_{k\leq p <i ; 1\leq q\leq j} u_{pq}\geq u_{k, n+3}>0}.$$
By the first assumption of Theorem \ref{eq61}, we have $\Gamma_{kj}(G_2)\leq \Gamma_{kj}(G_1)$ and $\Gamma_{1,i-1}(G_2)\leq \Gamma_{1,i-1}(G_1)$. Consider the depths of both $G_1$ and $G_2$, $\Gamma_{1,k-1}(G_2)=\Gamma_{1,k-1}(G_1)=0$. Thus \eqref{eq76} follows immediately from \eqref{eq761}.
\end{proof}

\begin{theorem}\label{thm2.4}
The map \eqref{eq58} is a bijection from the set of weighted graphs to ${\bf P}(\widetilde{\mc{A}}_{A_n})$. Furthermore, $\mathbb{I}(G_1)\leq\mathbb{I}(G_2)$ if and only if the condition \eqref{eq64} holds. 
\end{theorem} 

\begin{proof}
The ``if'' part of the second statement follows directly from Theorem \ref{eq61}. We prove the ``only if" part.

Let $G=(u_{ij})$ be a weighted graph. If $\mathbb{I}(G)$ is not indecomposable, then there exist $1\leq r<s<m<t\leq n+3$ such that $u_{rm}>0$, $u_{st}>0$. Let $G'=(v_{ij})$, $G''=(w_{ij})$ be two new weighted graphs such that
\earray{ll}{v_{ij}=w_{ij}=u_{ij}-1,&\text{ if } \{ij\}=\{rm\}\text{ or }\{st\};\\
                     v_{ij}=u_{ij}+1,~ w_{ij}=u_{ij}&\text{ if } \{ij\}=\{rt\} \text{ or }\{sm\};\\
                     v_{ij}=u_{ij}, ~ w_{ij}=u_{ij}+1&\text{ if } \{ij\}=\{rs\} \text{ or }\{tm\};\\
                     v_{ij}=w_{ij}=u_{ij},     &\text{ otherwise} .
}
Then by the Pl\"{u}cker relation \eqref{eq34}, $\mathbb{I}(G)=\mathbb{I}(G')+\mathbb{I}(G'')$. For any $\{kl\}$, 
\begin{equation}
I_{kl}(G)=\max\{I_{kl}(G'),~I_{kl}(G'')\}.
\end{equation}
If $\mathbb{I}(G')$, $\mathbb{I}(G'')$ are not indecomposable, repeat the above process. The product $\mathbb{I}(G)$ can be uniquely decomposed into a finite sum:
\begin{equation}
\mathbb{I}(G)=\sum_{i\in I}\mathbb{I}(L_i),
\end{equation}
where every $\mathbb{I}(L_i)$ is indecomposable, and 
\begin{equation}
I_{kl}(G)=\max_{i\in I}\{I_{kl}(L_i)\} ,~~~\forall \{kl\}.
\label{int}
\end{equation}

Let $G_1$, $G_2$ be two weighted graphs such that $\mathbb{I}(G_1)\leq\mathbb{I}(G_2)$. The decomposition of ${\Bbb I}(G_2)$ contains all indecomposable functions appearing in the decomposition of ${\Bbb I}(G_1)$.
By \eqref{int}, $I_{kl}(G_1)\leq I_{kl}(G_2)$. The condition $R_{i}(G_1)=R_{i}(G_2)$ follows by the same argument. Therefore the condition (\ref{eq64}) is necessary.

\vskip 3mm 

It is clear that the map \eqref{eq58} is surjective.
It remains to show that it is injective.
Notice that for each $G=(u_{ij})$, we have 
\be \la{6.10.21.28}
u_{ij}=\frac{1}{2}\big(I_{ij}(G)+I_{i-1,j-1}(G)-I_{i,j-1}(G)-I_{i-1,j}(G)\big).
\ee
If ${\Bbb I}(G_1)={\Bbb I}(G_2)$, then by condition (\ref{eq64}), $I_{kl}(G_1)=I_{kl}(G_2),~\forall \{kl\}$. Therefore by \eqref{6.10.21.28}, $G_1=G_2$. The Theorem is proved.
\end{proof}


\section{Cluster ensembles of type $A$}
\label{sec: 4}
\subsection{Definition of cluster ensembles}
\label{sec: 4.1}
For the convenience of the reader, we briefly recall the definition of cluster ensembles from \cite{FG}.
\begin{name}
A seed is a datum ${\bf i}=(I, I_0, \varepsilon, d)$, where $I$ is a finite set,  $I_0$ is a subset of  $I$, $\varepsilon=\{\varepsilon_{ij}\}$ is a $\mathbb{Z}$-valued function on $I\times I$,  and $d=\{d_i\}_{i\in I}$ is a set of positive rational numbers such that $\varepsilon_{ij}d_j^{-1}=-\varepsilon_{ji}d_i^{-1} $.
\end{name} 

Given a seed $\bf{i}$, any element $k\in I-I_0$ provides a new seed $\mu_k({\bf i})={\bf i}'=\{I', I'_0, \varepsilon', d'\}$ such that $I':=I$, $I'_0:=I'_0$, $d':=d$ and 
\eqarray{\varepsilon'_{ij}}{ll}{-\varepsilon_{ij},&\text{ if } k\in \{i,j\}\\
                                              \varepsilon_{ij},&\text{ if } \varepsilon_{ik}\varepsilon_{kj}\leq 0, \quad k\notin \{i, j\}\\
                                              \varepsilon_{ij}+|\varepsilon_{ik}|\cdot \varepsilon_{kj},     &\text{ if } \varepsilon_{ik}\varepsilon_{kj}> 0, \quad k\notin \{i, j\}\\
}
Here $\mu_k$ is called the seed mutation in the direction $k$. This mutation is involutive: $\mu_{k}^2({\bf i})={\bf i}$. Repeating the process in every direction for each new seed got via seed mutations, we obtain an \emph{n}-regular tree such that each of its vertices corresponds to a seed. Here \emph{n} is the cardinality of the set $I-I_0$. 

Now assign to each seed ${\bf i}$ two coordinate systems: $\mc{X}_{{\bf i}}=\{X_i~|~ i\in I\}$ and $\mc{A}_{{\bf i}}=\{A_i~|~i\in I\}$. There is a homomorphism $p$ relating $\mc{X}_{{\bf i}}$ and $\mc{A}_{{\bf i}}$:
\begin{equation} 
p^* X_i=\prod_{j\in I}A_j^{\varepsilon_{ij}}.
\label{mapp}\end{equation}
The transition maps between the coordinate systems assigned to ${\bf i}$ and ${\bf i}'=\mu_k({\bf i})$ are as follows:

\eqarray{\mu_k^*X_i'}{ll}{X_k^{-1}, &\text{ if } i=k\\
                                               X_i(1+X_k^{-{\rm sgn}(\varepsilon_{ij})})^{-\varepsilon_{ij}},&\text{ if } i\neq k,\\
                                              }
\eqarray{\mu_k^*A_i'}{ll}{A_k^{-1}(\prod_{j|\varepsilon_{kj}>0}A_j^{\varepsilon_{kj}}+\prod_{j|\varepsilon_{kj}<0}A_j^{-\varepsilon_{kj}}), &\text{ if } i=k\\
                                               A_i, &\text{ if } i\neq k,\\
                                              }
Clearly the transition maps are subtraction free and thus give rise to a pair of positive spaces. Denote it by $(\mc{X}_{\bf{|i|}}, \mc{A}_{\bf{|i|}})$ and call it a cluster ensemble.\\
 
Given a seed ${\bf i}=(I, I_0, \varepsilon, d)$, let ${\bf i}_0=(I-I_0, \emptyset, \varepsilon,  d)$, where $\varepsilon$ in ${\bf i}_0$ is the restriction of $\varepsilon$ in ${\bf i}$ to $(I-I_0)\times(I- I_0)$. Call ${\bf i}_0$ the reduced seed of $\bf{i}$. Let $(\mc{X}_{{|{\bf i}_0|}}, \mc{A}_{{|{\bf i}_0|}})$ be the cluster ensemble corresponding to the reduced seed ${\bf i}_0$. We have the following commutative diagram:
\be \la{eq78}
\begin{array}{cccc}
{\cal A}_{|{\bf i}|}
&\stackrel{p}{\lra} & {\cal X}_{|{\bf i}|}\\
&& \\
 ~\uparrow i &\searrow k  & ~\downarrow j\\
&& \\
{\cal A}_{|{\bf i}_0|}&\stackrel{p_0}{\lra} &{\cal X}_{|{\bf i}_0|} .
\end{array}
\ee

The maps $p, ~p_0$ are natural maps defined by \eqref{mapp}.

The map $i$ is an injective map such that $i^{*}A_i:= A_i$ if $i\in I-I_0$, otherwise $i^{*}A_i:= 1$.

The map $j$ is a surjective map such that $j^{*}X_i:= X_i$ for all $i\in I-I_0$.

The map $k$ is the composition of $p$ and $j$. Here $k$ is surjective if and only if the sub matrix $\varepsilon_{I-I_0, I}=(\varepsilon_{ij})$ is of full rank, where $(i,j)$ runs through $(I-I_0)\times I$.

\subsection{The map $k$}
In this section, we assume that the map \emph{k} is surjective.  It induces an injective linear map
\be
\tau_{\bf i}:~\Z^n\lra \Z^m, ~~(b_1,\ldots, b_n)\lms (a_1,\ldots, a_m). 
\ee
Here $n=\#(I-I_0)$, $m=\#I$, and $a_j=\sum_{i=1}^n b_i\varepsilon_{ij}$. For each $b=(b_1,\ldots,b_n)$, let $X^b:=\prod_{i=1}^n X_i^{b_i}$. By \eqref{mapp}, we have $k^*(X^b)=A^{\tau_{\bf i}(b)}.$ 

Let $\Q({\cal X}_{|{\bf i}_0|})$ be the field of rational functions on ${\cal X}_{|{\bf i}_0|}$.

\bl \la{5.8.4.20}
Let $f\in \Q({\cal X}_{|{\bf i}_0|})$.  Then $f\in {\Bbb L}_{+}({\cal X}_{|{\bf i}_0|})$ if and only if $k^*(f)\in {\Bbb L}_{+}({\cal A}_{|{\bf i}|})$.
\el

\begin{proof}
For each seed {\bf i}, let ${\cal X}_{{\bf i}_0}=\{X_i~|~i\in I-I_0\}$, ${\cal A}_{{\bf i}}=\{A_j~|~j\in I\}$ be corresponding coordinate systems. Let ${\Z}_{\geq 0}[X_i^{\pm}]$, ${\Z}_{\geq 0}[A_j^{\pm}]$ be the semirings of Laurent polynomials with non-negative integral coefficients. Since $k^*$ is injective and it maps Laurent monomials to Laurent monomials, we have
\be
f\in {\Z}_{\geq 0}[X_{i}^{\pm}]\Longleftrightarrow
k^*(f)\in {\Z}_{\geq 0}[ {A}_{j}^{\pm}].
\ee
By definition, we have ${\Bbb L}_{+}({\cal X}_{|{\bf i}_0|})=\bigcap_{{\bf i}_0}{\Z_{\geq 0}}[{X}_{i}^{\pm}]$, and ${\Bbb L}_{+}({\cal A}_{|{\bf i}|})=\bigcap_{{\bf i}}{\Z_{\geq 0}}[{ A}_{j}^{\pm}]$. The Lemma follows directly.
\end{proof}

\bl \la{6.8.10.06}
Let $f\in {\Bbb L}_{+}({\cal X}_{|{\bf i}_0|})$, $g\in {\Bbb L}_{+}({\cal A}_{|{\bf i}|})$ be such that $k^*(f)\geq g$. Then there exists a unique $g'\in {\Bbb L}_{+}({\cal X}_{|{\bf i}_0|})$ such that $k^*(g')=g.$
\el

\begin{proof}
Fix a seed ${\bf i}$. Let $f=\prod_{b\in \Z^n} c_bX^b$. Then 
\be
k^*(f)=\prod_{b\in \Z^n}c_b A^{\tau_{\bf i}(b)} :=\prod_{a\in \Z^m}c_a'A^a.
\ee
Here the number $c_a'>0$ implies that $a\in \tau_{\bf i}(\Z^n)$.

 Let $g=\prod_{a\in \Z^m}d_a A^a$. Since $k^*(f)\geq g$, then $c_a'\geq d_a$. If $d_a>0$, then $c_a'>0$ and thus $a\in \tau_{\bf i}(\Z^n)$. In other words, there exists a unique
\be
g'=\prod_{b\in\Z^n}d_{\tau_{\bf i}(b)}X^b \in \Z_{\geq 0}[X_j^{\pm}]
\ee
such that $k^*(g')=g$. 

By Lemma \ref{5.8.4.20}, $g'\in {\Bbb L}_{+}({\cal X}_{|{\bf i}_0|})$. The Lemma is proved.
 
\end{proof}

\bl \la{6.8.10.36}
Let $f\in \Q({\cal X}_{|{\bf i}_0|})$. Then $f\in {\bf E}({\cal X}_{|{\bf i}_0|})$ if and only if $k^*(f)\in {\bf E}({\cal A}_{|{\bf i}|})$.  
\el

\begin{proof}
If $f\in {\bf E}({\cal X}_{|{\bf i}_0|})$, then by Lemma \ref{5.8.4.20},  $k^*(f)\in {\Bbb L}_{+}({\cal A}_{|{\bf i}|})$. Here $k^*(f)$ must be indecomposable. Otherwise assume $k^*(f)=g+h$, where $g,h\in {\Bbb L}_{+}({\cal A}_{|{\bf i}|})$ and $g,h\neq 0$. By Lemma \ref{6.8.10.06}, there exist $g',h'\in {\Bbb L}_{+}({\cal X}_{|{\bf i}_0|})$ such that $g=k^*(g'),~h=k^*(h')$. Therefore $f=g'+h'$, which contradicts the assumption that $f$ is indecomposable.
The other direction follows similarly. The Lemma is proved.
\end{proof}

\bl \la{6.9.3.5}
For each $f\in {\bf P}({\cal X}_{|{\bf i}_0|})$, we have $k^*(f)\in {\bf P}({\cal A}_{|{\bf i}|})$. Moreover, it preserves the partial order structure:
\be \la{eq79}
\forall f,g\in {\bf P}({\cal X}_{|{\bf i}_0|}),~f\leq g \Longleftrightarrow k^*(f)\leq k^*(g).
\ee
\el
\begin{proof}
The first part follows from Lemma \ref{6.8.10.36}.
The second part follows from Lemma \ref{5.8.4.20}.
\end{proof}

\vskip 3mm

Let {\rm T} be a split algebraic torus. We define the group of characters
\be
X({\rm T}):={\rm Hom}({\rm T},{\Bbb G}_m).
\ee
In particular, each seed {\bf i} gives rise to a pair of tori
\be
{\rm T}_{{\bf i},{\cal A}}:={\rm Spec} (\C[A_j^{\pm}]),~{\rm T}_{{{\bf i}_0},{\cal X}}:={\rm Spec}(\C[X_i^{\pm}]); ~ j\in I,~i\in I-I_0.
\ee
Since $k$ is a surjective morphism from ${\rm T}_{{\bf i},{\cal A}}$ to ${\rm T}_{{\bf i}_0,{\cal X}}$, the group $X({\rm T}_{{\bf i}_0,{\cal X}})$ can be viewed as a sub lattice of $X({\rm T}_{{\bf i},{\cal A}})$. Let ${\rm Q}_{\bf i}:=X({\rm T}_{{\bf i},{\cal A}})/X({\rm T}_{{\bf i}_0, {\cal X}})$ be the corresponding quotient group. 

\vskip 3mm

Let ${\Bbb L}({\cal A}_{|{\bf i}|})$ be the coordinate ring of ${\cal A}_{|{\bf i}|}$. 
Each $f\in {\Bbb L}({\cal A}_{|{\bf i}|})$ can be uniquely expanded as
\be \la{6.8.3.58}
f= \sum_{[a]\in {\rm Q}_{\bf i} } f_{[a]}^{\bf i}
\ee
where $f_{[a]}^{\bf i}$ is a Laurent polynomial consisting of characters belong to the coset $[a]$. Namely, if $a\in X({\rm T}_{{\bf i}, {\cal A}})$ is a representative of $[a]\in {\rm Q}_{\bf i}$, then $f_{[a]}^{\bf i}=A^{a}\cdot k^*(g_a)$, where $g_a\in \Q({\cal X}_{|{\bf i}_0|})$.

\bl \la{6.8.4.04} For each pair $\overline{\bf i}, {\bf i}\in |{\bf i}|$, there is a  canonical isomorphism $\eta: {\rm Q}_{\overline{\bf i}}\stackrel{\sim}{\lra} {\rm Q}_{\bf i}$ such that
\be
f_{[a]}^{\overline{\bf i}}=f_{\eta([a])}^{\bf i},~~~\forall[a]\in {\rm Q}_{\overline{\bf i}},~~\forall f\in {\Bbb L}({\cal A}_{|\bf i|}).
\ee
\el

\begin{proof}
First assume $\overline{\bf i}=\mu_k({\bf i})$. By definition, we have 
\begin{align} \la{9.11.11.0}
\overline{A}_k&=\big(A_k^{-1}\prod_{j|\varepsilon_{kj}<0 }A_{j}^{-\varepsilon_{kj}}\big)\cdot k^*(1+X_k) \nonumber \\
\overline{A}_j&=A_j,~~\forall j\neq k.
\end{align}
We consider the following bijective map
$$
\eta: X({\rm T}_{\overline{\bf i}, {\cal A}})\lra X({\rm T}_{{\bf i}, {\cal A}}),~~~(a_1,\ldots, a_m)\lms (b_1,\ldots, b_m),
$$
such that
\eqarray{b_j}{ll}{-a_k,&\text{ if } j=k,\\
                                a_j-\varepsilon_{kj}\cdot a_k   ,&\text{ if } \varepsilon_{kj}< 0,\\
    a_j,  &\text{ otherwise.} \\
}
We pick a representative $a\in X({\rm T}_{\overline{\bf i},{\cal A}})$ for each $[a]\in {\rm Q}_{\overline{\bf i}}$. Then by \eqref{9.11.11.0}, we have
$$
f_{[a]}^{\overline{\bf i}}=\overline{A}^a\cdot k^*(g_a)=A^{\eta(a)}\cdot k^*\big((1+X_k)^{a_k}\cdot g_a\big),
$$
where $(1+X_k)^{a_k}\cdot g_a \in \Q({\cal X}_{|{\bf i}_0|})$. 

Notice that $\eta$ maps the sub lattice $X({\rm T}_{\overline{\bf i}_0,{\cal X}})$ onto $X({\rm T}_{{\bf i}_0, {\cal X}})$. It descends to a bijective map $\eta: {\rm Q}_{\overline{\bf i}}\lra {\rm Q}_{\bf i}$. We consider the expansion of $f$ under the seed ${\bf i}$:
$$
f=\sum_{\eta([a])\in {\rm Q}_{\bf i}}f_{\eta([a])}^{\bf i}=\sum _{\eta([a])\in {\rm Q}_{\bf i}} A^{\eta(a)}\cdot k^*\big((1+X_k)^{a_k}\cdot g_a\big).
$$
Since this expansion is unique, we have
$$
f_{\eta([a])}^{\bf i}=A^{\eta(a)}\cdot k^*\big((1+X_k)^{a_k}\cdot g_a\big)= f_{[a]}^{\overline{\bf i}}.
$$
The Lemma follows. For arbitrary seed $\overline{\bf i}$, it is a composition of mutations. Thus the Lemma is proved.
\end{proof}

\bl \la{6.8.16.58}
Let $f$ be as in (\ref{6.8.3.58}). If $f\in {\Bbb L}_{+}({\cal A}_{|{\bf i}|})$, then
\be
f_{[a]}^{\bf i}\in {\Bbb L}_{+}({\cal A}_{|{\bf i}|}), ~~\forall [a]\in {\rm Q}_{\bf i}.
\ee
\el

\begin{proof}
It follows immediately from Lemma \ref{6.8.4.04}.
\end{proof}

\bl \la{6.8.17.2}
If ${\bf E}({\cal A}_{|{\bf i}|}) $ is a basis of ${\Bbb L}({\cal A}_{|{\bf i}|})$, then ${\bf E}({\cal X}_{|{\bf i}_0|})$ is a basis of ${\Bbb L}({\cal X}_{|{\bf i}_0|}).$ 
\el
\begin{proof}

If $f\in {\bf E}({\cal A}_{|{\bf i}|})$, then there exists a unique $[a]\in {\rm Q}_{\bf i}$ such that $f=f_{[a]}^{\bf i}$, otherwise it is decomposable due to Lemma \ref{6.8.16.58}.

Let $g\in {\Bbb L}({\cal X}_{|{\bf i}_0|})$. Let $h=k^*(g)\in {\Bbb L}({\cal A}_{|{\bf i}|})$. Then by definition $h=h_{[0]}^{\bf i}$. 
Since ${\bf E}({\cal A}_{|{\bf i}|})$ is a basis of ${\Bbb L}({\cal A}_{|{\bf i}|})$, one has a unique decomposition
\be
h=\sum_i c_i h_i, ~~~h_i\in {\bf E}({\cal A}_{|{\bf i}|}).
\ee
Clearly here if $c_i\neq 0$, then $h_i=h_{i,[0]}^{\bf i}.$ Namely $h_i\in k^*(\Q({\cal X}_{|{\bf i}_0|}))$. Let $h_i=k^*(g_i)$. By Lemma \ref{6.8.10.36}, $g_i\in {\bf E}({\cal X}_{|{\bf i}_0|})$.  We thus get a unique decomposition: 
\be
g=\sum_i c_ig_i,~~~g_i\in {\bf E}({\cal X}_{|{\bf i}_0|}).
\ee 
The Lemma is proved.
\end{proof}

\subsection{The set of ${\cal A}$-laminations on a convex polygon}
Let us recall the definition of ${\cal A}$-laminations from \cite{FG1}. Most of the results of this section are from \cite{FG3} Section 3.
\bd [{\rm [\emph{loc.cit,} Definition 3.2]}]
\la{6.7.1.37}
An ${\cal A}$-lamination on a convex polygon is a collection of edges and mutually non-intersecting diagonals of the polygon with ${\Bbb A}$-valued weights, subjecting to the following conditions:
\begin{enumerate}
\item The weights of the diagonals are non-negative.
\item The sum of the weights of the diagonals and edges incident to a given vertex is zero.
\end{enumerate} 
\ed
Denote by ${\cal A}_L(n,{\Bbb A})$ the set of ${\cal A}$-laminations on a convex (\emph{n}+3)-gon. Recall that ${\Bbb A}$ can be $\Z, \Q,\R$. Clearly ${\cal A}_L(n,{\Bbb A})$ is a subset of weighted graphs. Let $G\in {\cal A}_L(n, {\Bbb A})$. Set
\be \la{6.9.4.04}
a_{kl}:=\frac{1}{2}I_{kl}(G).
\ee
Notice that $R_i(G)=0,~\forall i\in[1,n+3]$. By (\ref{eq63}), we have $a_{kl}=-\Gamma_{k+1,l}(G)\in {\Bbb A}$. It is easy to show that
\begin{align}
\la{6.7.2.28}
a_{ik}+a_{jl}&={\rm max}\{a_{ij}+a_{kl},~a_{il}+a_{jk}\},&&\text{if $i,j,k,l$ seated clockwise},\nonumber\\
a_{ij}&=0,&&\text{if $i,j$ are adjacent.}
\end{align}
Let \emph{T} be a triangulation of the (\emph{n}+3)-gon. 
\bl \la{6.7.4.15}
There is a bijection
\be
\phi_T: {\cal A}_L(n,{\Bbb A})\stackrel{\sim}{\lra}{\Bbb A}^{\{{\rm diagonals~of}~T\}},~~~G\lms \{a_{ij}(G)\}, ~\{ij\}\in T.
\ee
\el
\begin{proof}
We prove the Lemma by constructing the inverse map of $\phi_T$.

Let $\{a_{ij}\}\in {\Bbb A}^{\{\text{diagonals  of }T\}}$.  It can be uniquely extended to $\{a_{pq}\},~p,q\in [1,n+3]$ such that they satisfy (\ref{6.7.2.28}) and that $a_{pp}=0,~\forall p\in [1,n+3]$. Let 
\be \la{6.7.4.12}
u_{pq}=a_{q-1,p-1}+a_{pq}-a_{p,q-1}-a_{p-1,q}
\ee
We prove that $G=(u_{pq})$ is the pre-image of $\{a_{ij}\}$. We first prove that $G$ is an ${\cal A}$-lamination.

By the tropical Pl\"{u}cker relation, the weight $u_{pq}\geq 0$ for each diagonal $\{pq\}$. For each diagonal $\{pq\}$ such that $u_{pq}>0$, let $I=[p+1,q-1],~J=[q+1,n+3]\cup[1,p-1]$. The sum of weights of diagonals intersecting $\{pq\}$ is
\begin{align}
\sum_{i\in I,~j\in J}u_{ij}&=\sum_{j\in J}\big(\sum_{i\in I}(a_{ij}-a_{i-1,j})+\sum_{i\in I}(a_{i-1,j-1}-a_{i,j-1})\big)\nonumber\\
&=\sum_{j\in J}\big(a_{q-1,j}-a_{pj}+a_{p,j-1}-a_{q-1,j-1}\big)\nonumber\\
&=\sum_{j\in J}(a_{q-1,j}-a_{q-1,j-1})+\sum_{j\in J}(a_{p,j-1}-a_{p,j})\nonumber\\
&=a_{q-1,p-1}+a_{p,q}-a_{q-1,q}-a_{p,p-1}.
\end{align}
The tropical Pl\"{u}cker relation tells us that 
\be
\min\{u_{pq}, \sum_{i\in I, ~j\in J}u_{ij}\}=\min\{a_{q-1,p-1}+a_{p,q}-a_{q-1,q}-a_{p,p-1},~a_{q-1,p-1}+a_{pq}-a_{p,q-1}-a_{p-1,q}\}=0.
\ee
Since $u_{pq}>0$, one must have
\be
\sum_{i\in I, ~j\in J}u_{ij}=0.
\ee 
Each of $u_{ij}$ is non-negative. Therefore $u_{ij}=0,~\forall i\in I, j\in J$. Therefore ${G}$ is a collection of edges and mutually non-intersecting diagonals of the polygon.

Let $I=[k+1,l], J=[l+1,n+3]\cup[1,k]$. Similarly, we  have
\begin{align} \la{6.10.9.53}
a_{kl}(G)=\frac{1}{2}I_{kl}(G)=\frac{1}{2}\sum_{i\in I,~j\in J}u_{ij}=\frac{1}{2}(a_{lk}-a_{ll}+a_{kl}-a_{kk})=a_{kl}.
\end{align}
In particular $R_k(G)=I_{k,k+1}(G)=2a_{k,k+1}=0$,~$\forall k\in [1,n+3]$. Therefore $G$ is an ${\cal A}$-lamination. The map $\phi_T$ takes $G$ to $\{a_{ij}\}$. Hence $\phi_T$ is surjective. 

It is injective because every $G=(u_{pq})$ is uniquely determined by (\ref{6.7.4.12}). 
\end{proof}
By Lemma \ref{6.7.4.15} and the relations (\ref{6.7.2.28}), the following proposition is clear.
\bp{\rm (\cite{FG3})} \la{6.9.1.53}
There is a canonical isomorphism of sets
\be 
{\cal A}_{A_n}({\Bbb A}^t)={\cal A}_L(n,{\Bbb A})
\ee
such that 
\be
A_{ij}^t(l)=a_{ij}(l),~~~\forall ~{\rm diagonals}~ \{ij\}.
\ee
\ep

\subsection{Proof of Theorem \ref{mthm}}
\label{sec: 4.2}
Given a Cartan matrix of type $A_n$, we obtain a skewsymmetric matrix $\varepsilon$ by killing the 2's on the diagonal and changing signs under the diagonal. Let ${\bf i}=\{I, \emptyset, \varepsilon, d\}$ be such that $|I|=n$ and every $d_i\in d$ is 1. Denote by $\mc{X}_{A_n}$ the cluster $\mc{X}$-variety corresponding to this seed. Recall the cluster $\mc{A}$-variety $\widetilde{\mc{A}}_{A_n}$ from Definition \ref{eq35}. As shown in Section \ref{sec: 4.1}, there is a canonical surjective map $k: \widetilde{\mc{A}}_{A_n} \goto \mc{X}_{A_n}$. The space ${\cal X}_{A_n}$ is a partial completion of the moduli space ${\cal M}_{0,n+3}$. The next Lemma follows from \cite{FG3}, Section 3.
\bl \la{6.9.2.06}
For each weighted graph G,  ${\Bbb I}(G)\in k^*(\Q({\cal X}_{A_n}))$ if and only if 
\be
R_{i}(G)=0,~\forall i\in [1,n+3].
\ee
\el
\bt
There is a canonical isomorphism: ${\Bbb I}_{\cal A}: {\cal A}_{A_n}(\Z^t)\stackrel{\sim}{\lra} {\bf E}({\cal X}_{A_n})$. The set ${\bf E}({\cal X}_{A_n})$ is a basis of ${\Bbb L}({\cal X}_{A_n})$.
\et
\begin{proof}

The set ${\bf E}(\widetilde{\cal A}_{A_n})$ is a canonical basis of ${\Bbb L}(\widetilde{\cal A}_{A_n})$ ([FZ1, Section 4], \cite{C}). Thus the second part of our theorem follows from Lemma \ref{6.8.17.2}.

Moreover ${\bf E}(\widetilde{A}_{A_n})$ can be parametrized by the set of weighted graphs with mutually non-intersecting diagonals. 
By Lemma \ref{6.8.10.36} and Lemma \ref{6.9.2.06}, the set ${\bf E}({\cal X}_{A_n})$ is isomorphic to ${\cal A}_L(n,\Z)$. By Proposition \ref{6.9.1.53}, we construct a canonical isomorphism between ${\bf E}({\cal X}_{A_n})$ and ${\cal A}_{A_n}(\Z^t)$. The first part is proved.
\end{proof}

Now we prove Theorem \ref{mthm}.
\begin{proof}
 By Proposition \ref{6.9.1.53}, we  may replace ${\cal A}_{A_n}(\Z^t)$ by ${\cal A}_L(n,\Z)$. For each
 \be
 f=\prod_{i=1}^m {\Bbb I}_{\cal A}(l_i)=\sum_{l\in {\cal A}_L(n,\Z)}c(f;l){\Bbb I}_{\cal A}(l),
 \ee
 we have
 \be
 l\in S_f \Longleftrightarrow c(f;l)>0 \Longleftrightarrow f\geq {\Bbb I}_{\cal A}(l)\Longleftrightarrow k^*(f)\geq k^*\big({\Bbb I}_{\cal A}(l)\big).
 \ee
 The first two equivalences are by definition. The third one is by Lemma \ref{6.9.3.5}. By the construction of ${\Bbb I}_{\cal A}$, we have
 \be
 k^*({\Bbb I}_{\cal A}(l))={\Bbb I}(l),~~k^*(f)={\Bbb I}(\sum_{i=1}^m l_i).
  \ee
  Here ${\sum_{i=1}^m l_i}$ is the sum of the laminations $l_i$ as matrices. Notice that $l,l_1,\ldots, l_m$ are ${\cal A}$-laminations. Then
 $$
 R_p(l)=0,~~R_p(\sum_{i=1}^m l_i)=\sum_{1=1}^m R_p(l_i)=0,~~\forall p.
 $$
 The first condition of (\ref{eq64}) holds automatically. By Theorem \ref{thm2.4}, we get
 \be
 {\Bbb I}(l)\leq {\Bbb I}(\sum_{i=1}^m l_i) \Longleftrightarrow I_{jk}(l)\leq \sum_{i=1}^m I_{jk}(l_i),~\forall \{jk\}.
 \ee
By (\ref{6.9.4.04}) and Proposition \ref{6.9.1.53}, we have $A_{jk}^t(l)=\frac{1}{2}I_{jk}(l)$. Therefore
\be
 l\in S_f\Longleftrightarrow{\Bbb I}(l)\leq {\Bbb I}(\sum_{i=1}^m l_i) \Longleftrightarrow A_{jk}^t(l)\leq \sum_{i=1}^m A_{jk}^t(l_i),~\forall \{jk\}.
\ee
Then the support is
\be
S_f=\{l\in {\cal A}_{A_n}(\Z^t)~|~ A_{jk}^t(l)\leq \sum_{i=1}^m A_{jk}^t(l_i), ~\forall \{jk\}\}.
\ee
The first part of Theorem \ref{mthm} is proved. The second part follows from Theorem \ref{6.3.12.44}.
 \end{proof}

 A Stasheff polytope ${\cal F}_{c}^{\emptyset}$ is called regular if the defining set $c$ is a set of integers. 
 Denote by ${\rm St}(\mc{A}_{A_n})$ the set of regular Stasheff polytopes in  $\mc{A}_{A_n}(\mathbb{Z}^t)$. Define the set
 \be
 {\bf P}_{*}({\cal X}_{A_n}):=\{f\in {\Bbb L}_{+}({\cal X}_{A_n})~|~k^*(f)\in {\bf P}(\widetilde{\cal A}_{A_n})\}.
 \ee
 \begin{theorem}
 \label{thm: 4.2}
  The following map is a bijection:
 \be
 {\bf P}_{*}(\mc{X}_{A_n}) \lra {\rm St}(\mc{A}_{A_n}), ~~~ f \lms S_f.
 \ee
 Furthermore $f_1\leq f_2$ if and only if $S_{f_1}\subseteq S_{f_2}$.
 \end{theorem}
\begin{proof} 
For each $f\in {\bf P}_{*}({\cal X}_{A_n})$, by the above discussion, we have
$k^*(f)={\Bbb I}(G)$ such that
\be
S_f=\{l\in {\cal A}_{A_n}(\Z^t)~|~A_{jk}^t(l)\leq c_{jk}(G),~\forall \{jk\}\},
\ee
where
\be
c_{jk}(G):=\frac{1}{2}I_{jk}(G)=-\Gamma_{k+1,l}(G)\in \Z, ~~\forall \{jk\}.
\ee
Furthermore, the set $\{c_{jk}(G)\}$ satisfies (\ref{6.10.9.36}). By Theorem \ref{eq55}, $S_f\in {\rm St}({\cal A}_{A_n})$.

For $f_1={\Bbb I}(G_1),f_2={\Bbb I}(G_2)\in {\bf P}_{*}(\mc{X}_{A_n})$, following the same argument of the last proof, 
\be \la{6.11.1.1}
f_1\leq f_2 \Longleftrightarrow I_{jk}(G_1)\leq I_{jk}(G_2), ~\forall \{jk\}.
\ee
The right hand side is equivalent to $S_{f_1}\subseteq S_{f_2}$. The second part is proved.

Therefore if $S_{f_1}=S_{f_2}$, then $f_1=f_2$, so the map \eqref{6.11.1.1} is injective. 

Given a set $c=\{c_{ij}\}$ of integers satisfying (\ref{6.10.9.36}), let $u_{ij}=c_{ij}+c_{i-1,j-1}-c_{i, j-1}-c_{i-1, j}$, where $c_{ij}=0$ if $i, j$ are adjacent or $i=j$. Let $G=(u_{ij})$ be the corresponding weighted graph. Let $I=[k+1,l]$ and let $J=[l+1,n+3]\cup[1,k]$. By (\ref{6.10.9.53}),
\begin{align}
c_{kl}(G)=\frac{1}{2}I_{kl}(G)=c_{kl},~\forall \{kl\}.
\end{align}
In particular $R_i(G)=I_{i-1,i}(G)=0$ for all \emph{i}. By Lemma \ref{6.9.2.06},  there is $f\in {\bf P}_{*}(\mc{X}_{A_n})$ such that $k^*(f)=\mathbb{I}(G)$, and $S_f={\cal F}_{c}^{\emptyset}$. Thus the map is surjective.
\end{proof}

\subsection{Further discussion and conjecture}
\label{not}
\label{sec: 4.3}
The Minkowski sum gives rise to a new family of convex polytopes in tropical positive spaces.
\bd
\la{eq88}
 A subset $S$ of a tropical positive space $\mc{A}(\mathbb{A}^t)$ is called a Minkowski polytope if there exist finitely many points $l_1,...l_m\in \mc{A}(\mathbb{A}^t)$ such that $S$ is the Minkowski sum of these $l_i$'s:
\be
\la{eq89}
S=\{x~|~ F^t(x)\leq \sum_{i=1}^mF^t(l_i) \text{ for all F }\in \exset{A}\}.
\ee
Given any set $T \subset \exset{A}$, the $T$-face of $S$ is the following set
\be
\la{eq90}
\mc{F}^T=S\bigcap \{x~|~ F^t(x)=\sum_{i\in I}F^t(l_i) \text{ for all F }\in T\}.
\ee
\ed

${\bf Remark}.$ Not all single point sets in $\mc{A}(\mathbb{A}^t)$ are convex. For example, it follows from \cite{SZ} that for rank 2 cluster $\mc{A}$-varieties of affine type, point sets corresponding to imaginary roots are not convex. Therefore they are not Minkowski polytopes. In this case, the supports of products of indecomposable functions are not necessarily convex either. For example, let ${\delta}$ be the imaginary root. For each $n\geq 1$, there is an indecomposable function $z_n$ corresponding to $n\delta$. By [\emph{loc.cit.}, Proposition 5.4], for all $p\geq n\geq 1$,
\eqarray{z_pz_n}{ll}{z_{p-n}+z_{p+n},~~& {\rm if~}p>n;\\
2+z_{2n},~~&{\rm if~}p=n.
}
The support of $z_pz_n$ is clearly not convex.

\vskip2mm

Assuming the Duality Conjecture, a generalization of Theorem \ref{mthm} is as follows.
\bcon \la{eq91} Let $f,g\in \mathbb{L}_{+}(\mc{X})$ be two universally positive Laurent polynomials on a cluster $\mc{X}$-variety. If both supports $S_f$ and $S_g$ are Minkowski polytopes in $\tropsp{A^\vee}$, then the support $S_{fg}$ is the Minkowski sum of $S_f$ and $S_g$. 
\econ


\vskip 2mm
Department of Mathematics, Yale University, New Haven, CT 06511, United States

 {\it E-mail address}: \url{linhui.shen@yale.edu}
 
\end{document}